\documentclass[10pt]{amsart}

\usepackage{amsthm, amsfonts, amssymb, color}
 \usepackage{mathrsfs}
\usepackage{amsmath}
 \usepackage{amstext, amsxtra}
  \usepackage{txfonts}
 \usepackage[colorlinks, linkcolor=black, citecolor=blue, pagebackref, hypertexnames=false]{hyperref}
\usepackage{graphicx}
\usepackage{overpic}
\usepackage{CJK}
\usepackage{slashed}

 \allowdisplaybreaks
\newtheorem{theorem}{Theorem}[section]
\newtheorem{proposition}[theorem]{Proposition}
\newtheorem{coro}[theorem]{Corollary}
\newtheorem{lemma}[theorem]{Lemma}
\newtheorem{definition}[theorem]{Definition}

\newtheorem{remark}[theorem]{Remark}

\renewcommand\Im{\operatorname{Im}}

\title[Scattering]{
Scattering and blowup for $L^{2}$-supercritical and $\dot{H}^{2}$-subcritical biharmonic NLS with potentials}
\author{Qing Guo, \ Hua Wang \ and Xiaohua Yao}

\address {Qing Guo, College of Science,
 Minzu University of China, Beijing, 100081, P.R. China}
\email{guoqing0117@163.com}
\address{ Hua Wang, School of Mathematics and Statistics and Hubei Province Key Laboratory of Mathematical Physics,
Central China Normal University, Wuhan, 430079, P.R. China}
\email{wanghua\_math@126.com}
\address{Xiaohua Yao, School of Mathematics and Statistics and  Hubei Province Key Laboratory of Mathematical Physics,
 Central China Normal University, Wuhan, 430079, P.R. China}
\email{yaoxiaohua@mail.ccnu.edu.cn }
\date{\today}

\begin{document}

\maketitle
\begin{abstract}
We mainly consider the focusing  biharmonic Schr\"odinger
equation with a large radial repulsive potential $V(x)$:
\begin{equation*}
\left\{ \begin{aligned}
  iu_{t}+(\Delta^2+V)u-|u|^{p-1}u=0,\;\;(t,x) \in {{\bf{R}}\times{\bf{R}}^{N}}, \\
  u(0, x)=u_{0}(x)\in H^{2}({\bf{R}}^{N}),
 \end{aligned}\right.
 \end{equation*}
If $N>8$, \  $1+\frac{8}{N}<p<1+\frac{8}{N-4}$ (i.e. the $L^{2}$-supercritical and $\dot{H}^{2}$-subcritical case ),  and $\langle x\rangle^\beta \big(|V(x)|+|\nabla V(x)|\big)\in L^\infty$ for some $\beta>N+4$,  then we firstly prove a global well-posedness and scattering result for
the radial  data $u_0\in H^2({\bf R}^N)$ which satisfies that   $$ M(u_0)^{\frac{2-s_c}{s_c}}E(u_0)<M(Q)^{\frac{2-s_c}{s_c}}E_{0}(Q)
\ \ {\rm{and}}\ \  \|u_{0}\|^{\frac{2-s_c}{s_c}}_{L^{2}}\|H^{\frac{1}{2}}
u_{0}\|_{L^{2}}<\|Q\|^{\frac{2-s_c}{s_c}}_{L^{2}}\|\Delta Q\|_{L^{2}}, $$  where $s_c=\frac{N}{2}-\frac{4}{p-1}\in(0,2)$, $H=\Delta^2+V$ and $Q$ is the ground state of
$\Delta^2Q+(2-s_c)Q-|Q|^{p-1}Q=0$.

 We crucially establish full Strichartz estimates and smoothing estimates of linear flow with a large poetential $V$, which are fundamental to our scattering results.

 Finally, based on the method introduced in \cite[T. Boulenger, E. Lenzmann, Blow up for biharmonic NLS, Ann. Sci. $\acute{E}$c. Norm. Sup$\acute{e}$r., 50(2017), 503-544]{B-Lenzmann},
we also prove a blow-up result for a class of potential $V$ and
the radial  data $u_0\in H^2({\bf R}^N)$ satisfying that   $$ M(u_0)^{\frac{2-s_c}{s_c}}E(u_0)<M(Q)^{\frac{2-s_c}{s_c}}E_{0}(Q)
\ \ {\rm{and}}\ \  \|u_{0}\|^{\frac{2-s_c}{s_c}}_{L^{2}}\|H^{\frac{1}{2}}
u_{0}\|_{L^{2}}>\|Q\|^{\frac{2-s_c}{s_c}}_{L^{2}}\|\Delta Q\|_{L^{2}}. $$
\end{abstract}

\tableofcontents

\section{Introduction}
\setcounter{equation}{0}
In this paper, we consider the biharmonic NLS with a potential ($\rm{BNLS_{V}}$)
\begin{equation}\label{1.1}
\left\{ \begin{aligned}
  iu_{t}+Hu+\lambda|u|^{p-1}u=0,\;\;(t,x) \in {{\bf{R}}\times{\bf{R}}^{N}}, \\
  u(0, x)=u_{0}(x)\in H^{2}({\bf{R}}^{N}),
 \end{aligned}\right.
 \end{equation}
where $u: I\times {\bf R}^{N}\rightarrow {\bf C}$ is a complex-valued function, $H=H_{0}+V$,
$H_{0}=\Delta^{2}$, $V: {\bf R}^{N}\rightarrow {\bf R}$, $\lambda=\pm 1$ and
$1<p<\infty$. The defocusing regime corresponds to the case $\lambda=+1$, and the focusing regime
to the case $\lambda=-1$. The biharmornic Schr\"{o}dinger equation has been introduced by Karpman \cite{Karp} and
Karpman and Shagalor \cite{Karp-Sha} to take into account the role of small fourth order dispersion terms in the
propagation of intense laser beams in a bulk medium with kerr nonlinearity.
The equation \eqref{1.1} has two important conservation laws in the energy space
$H^{2}({\bf{R}}^{N})$: The mass is defined by
\begin{align}\label{1.2}
M(u)=\displaystyle\int_{{\bf R}^{N}}|u(x)|^{2}dx,
\end{align}
and the energy is defined by
\begin{align}\label{1.2}
E(u)=E_{V}(u)=\frac{1}{2}\displaystyle\int_{{\bf R}^{N}}|\Delta u(x)|^{2}dx
+\frac{1}{2}\displaystyle\int_{{\bf R}^{N}}V(x)|u(x)|^{2}dx
+\frac{\lambda}{p+1}\displaystyle\int_{{\bf R}^{N}}|u(x)|^{p+1}dx.
\end{align}
When $V$ vanishes, we replace $E(u)$ by $E_{0}(u)$.
Moreover, you can easily see that the equation \eqref{1.1}
without potentials is invariant under the scaling transformation $u(x,t)\rightarrow l^{\frac{1}{p-1}}
u(l x, l^{2}t)$, which also leaves the norm of the homogeneous Sobolev space
$\dot{H}^{s_{c}}({\bf R}^{N})$ invariant, where $s_{c}=\frac{N}{2}-\frac{4}{p-1}$.
So we call that the equation \eqref{1.1} is energy subcritical for $n\leq 4$ or $p<1+\frac{8}{N-4}$
when $n\geq 5$, which correspond to $s_c<2$. Energy-criticality appears with the power
$p=1+\frac{8}{N-4}$, corresponding to $s_c=2$, and mass-criticality with power $p=1+\frac{8}{N}$ when $s_c=0$.

Let's recall some progress on the global well-posedness and scattering to \eqref{1.1} when $V=0$.
Fibich, Ilan and Papanicolaou \cite{Fibich} describe various properties of the equation in the subcritical
regime, with part of their analysis relying on very interesting numerical developments. Segata in
\cite{Seg} proved scattering  for the cubic nonlinearity in ${\bf R}$; while in higher dimensions
$5\leq N\leq 8$, the scattering results in $H^{2}({\bf R}^{N})$ were obtained by Pausader in \cite{Pau2},
which was extended by Miao, Xu and Zhao in \cite{MXZ3} to a low regularity space $H^{s}({\bf R}^{N})$
with some $s<2$ for $5\leq N\leq 7$.
Global well-posedness and scattering for the energy critical case were considered by Miao, Xu and Zhao
in \cite{MXZ1}, \cite{MXZ2} and Pausader in \cite{Pau1} and \cite{Pau3}. In \cite{Pau-Shao}, Pausader and Shao proved that scattering
 for the mass-critical fourth-order
Sch\"{o}dinger equation holds true in $L^{2}({\bf R}^{N})$ in high dimensions $N\geq 5$.
As for the mass-supercritical and energy-subcritical case, that is with the power $1+\frac{8}{N}<p<1+\frac{8}{N-4}$ $(N\geq 5)$,
the scattering results for the
defocusing case ($\lambda=+1$)  in the energy
space could be obtained using the argument in Lin and Strauss \cite{Lin} as discussed in \cite{Pau3}, also in \cite{Caz}. The same results
were established in \cite{Pau-Xia} for low dimensions $1\leq N\leq 4$ and $1+\frac{8}{N}<p<\infty$.

While for the corresponding focusing case ($\lambda=-1$), the first author \cite{Guo} recently
obtained a mass-supercritical and energy-subcritical scattering result with radial initial data for
all dimensions. Note that when $\lambda=-1$, one cannot hope to get a similar global result as in
\cite{Pau3}. Indeed, the existence of a nontrivial solution of the elliptic equation
\begin{align}\label{1.4}
\Delta^{2}Q+(2-s_{c})Q-|Q|^{p-1}Q=0
\end{align}
which we refer to as the ground state $Q\in H^{2}({\bf R}^{N})$, can be obtained by similar method to that
used in \cite{B-Lenzmann}. We then conclude that solitary waves $u(x, t)=e^{i(2-s_{c})t}Q(x)$ do not scatter. One
can refer to \cite{Fibich} for some similar results. The first author obtained the following result of scattering
for the solution of \eqref{1.1} with $V=0$ and radial data, which would complement the recent analysis on blowup theory by
Boulenger and Lenzmann \cite{B-Lenzmann}.

\begin{theorem}\label{th1.1} ( See \cite{B-Lenzmann, Guo} )
Assume that $V=0$, $\lambda=-1$, $1+\frac{8}{N}<p<1+\frac{8}{N-4}$
(when $2\le N\leq 4$, $1+\frac{8}{N}<p<\infty$).
Let $u_{0}\in H^{2}({\bf R}^{N})$ be radial and $u\in C(I;H^2({\bf R}^N))$ be the corresponding solution to \eqref{1.1} with
maximal forward time interval of existence $I\subset {\bf R}$. Then

(i) If
\begin{align}\label{1.5}
M(u_{0})^{\frac{2-s_{c}}{s_{c}}}E_{0}(u_{0})<M(Q)^{\frac{2-s_{c}}{s_{c}}}E_{0}(Q),
\end{align}
and
\begin{align}\label{1.6}
\|u_{0}\|_{L^{2}({\bf R}^{N})}^{\frac{2-s_{c}}{s_{c}}}\|\Delta u_{0}\|_{L^{2}({\bf R}^{N})}
<\|Q\|_{L^{2}({\bf R}^{N})}^{\frac{2-s_{c}}{s_{c}}}\|\Delta Q\|_{L^{2}({\bf R}^{N})},
\end{align}
where $Q$ is the solution of \eqref{1.4}, then $I=(-\infty, +\infty)$, and $u$ scatters in $H^{2}({\bf R}^{N})$.
That is, there exists $\phi_{\pm}\in H^{2}({\bf{R}}^{N})$ such that
\begin{align}\label{1.7}
\lim_{t\rightarrow \pm\infty}\|u(t)-e^{itH_{0}}\phi_{\pm}\|_{H^{2}({\bf R}^{N})}=0.
\end{align}

(ii) Either if $E_{0}(u_{0})<0$ or, if $E_{0}(u_{0})\ge 0$,  assume that  \eqref{1.5} and
\begin{align}\label{1.6'}
\|u_{0}\|_{L^{2}({\bf R}^{N})}^{\frac{2-s_{c}}{s_{c}}}\|\Delta u_{0}\|_{L^{2}({\bf R}^{N})}
>\|Q\|_{L^{2}({\bf R}^{N})}^{\frac{2-s_{c}}{s_{c}}}\|\Delta Q\|_{L^{2}({\bf R}^{N})}
\end{align}
hold, then the solution $u\in C([0,T);H^2({\bf R}^N))$ of \eqref{1.1} blows up in finite time, i.e., there exists some $0<T<+\infty$ such that
$\lim_{t\uparrow T}\|\Delta u(t)\|_{L^2}=+\infty$.
\end{theorem}


Motivated by these works,
 we naturally hope to extend Theorem \ref{th1.1} above in Q. Guo\cite{Guo} and Boulenger and Lenzmann \cite{B-Lenzmann}
to the case with a potential$V$, that is, to get the scattering and blow-up results in the energy space for the focusing ${\rm BNLS_{V}}$ \eqref{1.1}. For the end, however, there are several crucial obstacles due to the existence of  potential $V$.

Firstly, we need to establish the Strichartz estimates of linear group $e^{it(\Delta^2+V)}$, which are fundamental to the nonlinear equation ${\rm BNLS_{V}}$ \eqref{1.1}. Recall that in the free biharmonic operator $\Delta^2$, Ben-Artzi, Koch and Saut \cite{BKS} had proven the following sharp kernel estimate, \begin{equation}\label{freeL1Linfityepsilon}
  |D^{\alpha}I_{0}(t,x)|\leq C|t|^{-(N+|\alpha|)/4}\ \Big(1+|t|^{-1/4}|x|\Big)^{(|\alpha|-N)/3},\,\, t\neq0,\,x\in\mathbf{R}^{N},
\end{equation}
where $I_{0}(t,x)$ is the kernel of $e^{it\Delta^2}$. The above estimate implies the $L^{1}\rightarrow L^{\infty}$-estimate of $e^{it\Delta^2}$, namely
\begin{equation}\label{freeL1Linfty}
  \|D^\alpha e^{it\Delta^2}\|_{L^{1}(\mathbf{R}^{N})\rightarrow L^{\infty}(\mathbf{R}^{N})}\le C |t|^{-(N+|\alpha|)/4}, \ t\neq0, \ |\alpha|\le N.
\end{equation}
Hence the endpoint Strichartz estimates for the free group  $e^{it\Delta^2}$ can be established by  using the $L^{1}\rightarrow L^{\infty}$ estimate \eqref{freeL1Linfty} and  Keel-Tao arguments ( See \cite{Keel} ). For instance,  by \eqref{freeL1Linfty} we can establish that for any {\it S-admissible pairs} $(q,r)$
and $(a,b)$, and any $s\geq0$,
 \begin{equation}\label{2.3}
\Vert |\nabla|^{s}u\Vert_{L^q(I,L^r)}\leq C\  \Big(\Vert
|\nabla|^{s-\frac2q}u_0\Vert_{L^2}+ \Vert
|\nabla|^{s-\frac2q-\frac2a}h\Vert_{L^{a'}(I,L^{b'})}\Big),
\end{equation}
 and so on, where  $u$ is the  solution given by
 \begin{equation}\label{2.1'}
 u(t)=e^{it\Delta^2}u_0+i\int_0^te^{i(t-s)\Delta^2}h(s)ds.
\end{equation}
Indeed, one can see  more refined Strichartz estimates with regularity  in Section \ref{sec-2} below. For the fourth-order Schr\"odinger operator $H=\Delta^2+V$, it is much difficult to establish the similar kernel estimate \eqref{freeL1Linfityepsilon} for $e^{itH}$,  and  hard to prove the $L^{1}\rightarrow L^{\infty}$-estimate \eqref{freeL1Linfty}. In order to obtain Strichartz estimates of $e^{itH}$, we will use  Jensen-Kato decay estimate and local decay estimate of $H$ to overcome the difficulties caused by the potential $V$.

Very recently, Feng, Soffer and Yao \cite{FSY} have firstly established the following Jensen-Kato type decay estimate of the fourth order Schr\"odinger operator $H=\Delta^2+V$ ( see Lemma \ref{lemFSY} below ):
\begin{equation}\label{JKd}
  \|\langle x\rangle^{-\sigma} e^{-itH}P_{ac} \langle x\rangle^{-\sigma}\|_{L^{2}(\mathbf{R}^{N})-L^{2}(\mathbf{R}^{N})}\leq C \  \langle t\rangle^{-N/4}, \ t\in\mathbf{R},\ \sigma>N/2+2,
\end{equation}
under the assumptions that $\langle x\rangle^{\beta}V(x)\in L^{\infty}(\mathbf{R}^{N})$ for some large $\beta>0$, and $H$ has no positive embedded eigenvalues and  0 is not an eigenvalue nor resonance of $H$. Here $P_{ac}$ denotes the projection onto the absolutely continuous spectrum space of $H$, which removes the eigenstates and is necessary to dispersive estimate of $e^{itH}$.
We remark that Kato-Jensen type estimates is original in Jensen and Kato' famous work \cite{JK} for Schr\"odinger operator $-\Delta+V$ , which since later plays key roles in many important problems, such as $L^p$-decay estimates of Schr\"odinger operator in \cite{JS}, Soliton stability of NLS in \cite{BP1}, and so on.

In this paper, we will used Kato-Jensen estimates \eqref{JKd} to establish several useful Strichartz estimates and
 smoothing estimates for the linear solution $e^{itH}$ of \eqref{1.1}, {\it under the helps of some further conditions on $V$ and the restriction of dimension $N$}. Here, we do not attempt to express these specific  Strichartz estimates with potential. One can see Proposition \ref{lem2.1}, Proposition \ref{lem2.3} and Proposition \ref{lem2.4} in Section 2 below.
  Finally, we mention
that some Strichartz type estimates obtained here are independent of the scaling $V(x)\rightarrow
V_{r}=\frac{1}{r^{4}}V(\frac{x}{r})$ for any given $r>0$, which is very important to establish
linear profile decomposition with a potential (see Proposition
\ref{pro5.3} below) .


%
%
%

Our first scattering result in this paper can be stated as follows:

\begin{theorem}\label{th1.2}
 Let  $V$ be a radial real $C^1$-function of ${\bf R}^{N}$ satisfying that $x\cdot \nabla V\leq 0$ and
 $$|V(x)|+|\nabla V(x)|\le C (1+ |x|)^{-\beta}
$$ for some $\beta>N+4$.
Suppose that $\lambda=-1$,\ $N>8$, \ $1+\frac{8}{N}<p<1+\frac{8}{N-4}$, \
$u_{0}\in H^{2}({\bf R}^{N})$ is radial and $u\in C(I;H^2({\bf R}^N))$ is the corresponding solution to \eqref{1.1} with
maixmal forward time interval of existence $I\subset {\bf R}$.
If
\begin{align}\label{1.8}
M(u_{0})^{\frac{2-s_{c}}{s_{c}}}E(u_{0})<M(Q)^{\frac{2-s_{c}}{s_{c}}}E_{0}(Q),
\end{align}
and
\begin{align}\label{1.9}
\|u_{0}\|_{L^{2}({\bf R}^{N})}^{\frac{2-s_{c}}{s_{c}}}\|H^{\frac{1}{2}} u_{0}\|_{L^{2}({\bf R}^{N})}
<\|Q\|_{L^{2}({\bf R}^{N})}^{\frac{2-s_{c}}{s_{c}}}\|\Delta Q\|_{L^{2}({\bf R}^{N})}
\end{align}
where $Q$ is the solution of \eqref{1.4} and $H=\Delta^2+V$, then $I=(-\infty, +\infty)$, and $u$ scatters in $H^{2}({\bf R}^{N})$.
That is, there exist $\phi_{\pm}\in H^{2}({\bf{R}}^{N})$ such that
\begin{align}\label{1.10}
\lim_{t\rightarrow \pm\infty}\|u(t)-e^{itH}\phi_{\pm}\|_{H^{2}({\bf R}^{N})}=0.
\end{align}
\end{theorem}

Some comments on the conditions on $V$ and results of Theorem \ref{th1.2} as follows:

\begin{remark}\label{rem1.5}
There exist a great number of potentials $V$ satisfying the conditions in Theorem \ref{th1.2} above. For instance, simply, for any $\sigma>N/2+2$, we can take
$$V(x)=\frac{C}{(1+|x|^2)^{\sigma}}, \ C>0.$$
The radial requirement of $V$ comes from the focusing case( $\lambda=-1$ ), which can be removed in the following defousing case( $\lambda=1$ ). The decay index $\beta$ and restriction of dimension $N$ are technical and at present not optimal, which surly can be improved further. The repulsive condition $(x\cdot\nabla) V(x)\le 0$ plays an important role in the spectrum of $ H=\Delta^2+V$ and Morawetz estimates in this paper. In particular, the repulsive condition can be used to show that $H$ has no any eigenvalue in $\mathbf{R}$, see Section \ref{sec-2} below.

\end{remark}

\begin{remark}\label{rem1.3}
Note that both $(x\cdot\nabla) V(x)\le 0$ and $\lim_{x\rightarrow\infty}V(x)=0$, imply $V(x)\ge0$. In fact, it can be easily concluded by the following integral
$$V(x)=-\int^\infty_1\frac{d}{ds}\big(V(sx)\big)ds\ge0, \ \ x\neq0,$$
where $\frac{d}{ds}\big(V(sx)\big)=\frac{1}{s}(sx\cdot \nabla)V(sx)\le 0$.  Thus, $H=\Delta^2+V$ is a nonnegative self-adjoint operator, and
$$\|H^{\frac{1}{2}} u_{0}\|^2_{L^{2}}=\langle Hu_0, u_0\rangle=\int_{{\bf R}^{N}}|\Delta u_0|^2dx+\int_{{\bf R}^{N}}V|u_0|^2dx.$$
In particular, if $0\leq V\in L^{\frac{N}{4}}({\bf R}^{N}), N>4$, then we have
$$
\|H^{\frac{1}{2}}f\|_{L^{2}}\sim \int_{{\bf R}^{N}}|\Delta u_0|^2dx=\|\Delta u_0\|_{L^{2}}^{2}.
$$
Indeed, since $V\ge 0$, clearly
$
\|\Delta u_0\|_{L^{2}}^{2}\leq \|H^{\frac{1}{2}}u_0\|_{L^{2}}^{2}.
$
On the other hand, since $V\in  L^{\frac{N}{4}}$,  it follows from  H\"older inequality and Sobolev embedding that
$$
\|H^{\frac{1}{2}}u_0\|_{L^{2}}^{2}\leq \|\Delta u_0\|_{L^{2}}^{2}+\|V\|_{L^{\frac{N}{4}}}\|u_0\|_{L^{\frac{2N}{N-4}}}^{2}
\lesssim \|\Delta u_0\|_{L^{2}}^{2}.
$$

\end{remark}

Using the Morawetz estimates, Feng, the second and third authors \cite{FWY} considered the small potential $V$ when $N\geq 7$ for the defocusing $\rm{BNLS_{V}}$
\eqref{1.1} with non-radial initial data.  Based on the new Strichartz estimates with large potentials,  in this paper we can extend the result in \cite{FWY} to
the large potential case in $N>8$ (As the proof is almost the same as the one in \cite{FWY}, it will be omitted here).

\begin{theorem}\label{th1.6}
 Let $\lambda=+1$, $ N>8$, $1+\frac{8}{N}<p<1+\frac{8}{N-4}$ and $V$ be a real $C^1$-function of ${\bf R}^{N}$ satisfying that $x\cdot \nabla V\leq 0$ and
 $$|V(x)|+|\nabla V(x)|\le C (1+ |x|)^{-\beta}
$$ for some $\beta>N+4$. Assume that $u_{0}\in H^{2}({\bf R}^{N})$ and $u\in C(I;H^2({\bf R}^N))$ be the corresponding solution to \eqref{1.1} with
maixmal forward time interval of existence $I\subset {\bf R}$.
Then $I=(-\infty, +\infty)$, and $u$ scatters in $H^{2}({\bf R}^{N})$.
\end{theorem}

Finally, we turn to state our blow-up result. Note that Boulenger and Lenzmann in \cite{B-Lenzmann} have utilized the (localized) Riesz bivariance to get blow-up for biharmonic NLS, then
we can  apply their method to the Biharmonic NLS equation with certain large potential $V$.  For the end, in the following blowup result,
 {\it we assume that $V$ is a nonnegative radial real $C^1$-function of ${\bf R}^{N}$ satisfying $|\nabla V(x)|\le C (1+ |x|)^{-1}$, and set
\begin{align*}
W(x):=4V(x)+x\cdot\nabla V(x)=W_+(x)-W_-(x),
\end{align*}
where $W_\pm(x)$ denote the positive part and negative part of $W(x)$.}

\begin{theorem}\label{thblowup}Suppose that $\lambda=-1$,\ $N\geq 5$, \ $1+\frac{8}{N}<p<1+\frac{8}{N-4}$,  $u_{0}\in H^{2}({\bf R}^{N})$ is radial and $u\in C(I;H^2({\bf R}^N))$ is the corresponding solution to \eqref{1.1} with
maximal forward time interval of existence $I\subset {\bf R}$. Let  $ W\geq0$
or $\|W_-\|_{L^{\frac N4}}$ be sufficiently small. If $u_0$ furthermore satisfies one of the following two conditions:
(i)  $E(u_0)<0$; and (ii) $E(u_0)\geq0$,
$$
M(u_{0})^{\frac{2-s_{c}}{s_{c}}}E(u_{0})<M(Q)^{\frac{2-s_{c}}{s_{c}}}E_{0}(Q),
$$
and
$$
\|u_{0}\|_{L^{2}({\bf R}^{N})}^{\frac{2-s_{c}}{s_{c}}}\|H^{\frac{1}{2}} u_{0}\|_{L^{2}({\bf R}^{N})}
>\|Q\|_{L^{2}({\bf R}^{N})}^{\frac{2-s_{c}}{s_{c}}}\|\Delta Q\|_{L^{2}({\bf R}^{N})},
$$
where $Q$ is the solution of \eqref{1.4} and $H=\Delta^2+V$, then the solution $u\in C([0,T);H^2({\bf R}^N))$ of \eqref{1.1} blows up in finite time, i.e., there exists some $0<T<+\infty$ such that
$\lim_{t\uparrow T}\|\Delta u(t)\|_{L^2}=+\infty$.
\end{theorem}

\begin{remark}\label{rem1.7}
That $\|W_-\|_{L^{\frac N4}}$ is sufficiently small, means that there exists some constant $\delta>0$ such that $\|W_-\|_{L^{\frac N4}}\le \delta$. Clearly, if $V$ is suitably small, then the whole $W$ is small. Nevertheless, the smallness of $V$ is not absolutely necessary to the blowup result above.
Combining with Theorem \ref{thblowup},  we remark that the condition
$$
\|u_{0}\|_{L^{2}({\bf R}^{N})}^{\frac{2-s_{c}}{s_{c}}}\|H^{\frac{1}{2}} u_{0}\|_{L^{2}({\bf R}^{N})}
<\|Q\|_{L^{2}({\bf R}^{N})}^{\frac{2-s_{c}}{s_{c}}}\|\Delta Q\|_{L^{2}({\bf R}^{N})},
$$
is sharp for scattering result in Theorem \ref{th1.2}.
\end{remark}

In the sequel, we only consider the case $\lambda=-1$.
This present paper is organized as follows. We fix notations at the end of Section 1. In Section 2,
We establish some Strichartz type estimates, upon
which we obtain linear scattering. In Section 3, we establish local theory, the small data scattering and the perturbation theory.
The variational structure of the ground state of an elliptic problem is given in Section 4.
In Section 5 we prove a dichotomy proposition of global well-posedness versus blowing up, which
yields the comparability of the total energy and the gradient.
The concentration compactness principle is used in Section 6 to give a critical element, which
yields a contradiction through a virial-type estimate in Section 7, concluding the proof of Theorem \ref{th1.2}. In Section 8, we prove
the blow-up results, based on the argument of Boulenger and Lenzmann \cite{B-Lenzmann}.

\textbf{Notations:}:

we fix notations used throughout the paper. In what follows, we write $A\lesssim B$ to signify
that there exists a constant $C$ such that $A\leq CB$. And we denote $A\sim B$ when
$A\lesssim B\lesssim A$.

Let $L^{q}=L^{q}({\bf R}^{N})$ be the usual Lebesgue spaces, and $L_{I}^{q}L_{x}^{r}$ or $L^q(I,L^r)$ be the space
of measurable functions from an interval $I\subset {\bf R}$ to $L_{x}^{r}$ whose $L_{I}^{q}L_{x}^{r}$-
norm $
\|\cdot\|_{L_{I}^{q}L_{x}^{r}}$ is finite, where
\begin{align}\label{1.11}
\|u\|_{L_{I}^{q}L_{x}^{r}}=\Big(\displaystyle\int_{I}\|u(t)\|_{L_{x}^{r}}^{q}dt\Big)^{\frac{1}{r}}.
\end{align}
When $I={\bf R}$ or $I=[0, T]$, we may use $L_{t}^{q}L_{x}^{r}$ or $L_{T}^{q}L_{x}^{r}$ instead of
$L_{I}^{q}L_{x}^{r}$, respectively. In particular, when $q=r$, we may simply write them as
$L_{t,x}^{q}$ or $L_{T,x}^{q}$, respectively.

Moreover, the Fourier transform on ${\bf R}^{N}$ is defined by
$\hat{f}(\xi)=(2\pi)^{-\frac{N}{2}}\displaystyle\int_{{\bf R}^{N}} e^{-ix\cdot\xi}f(x)dx$.
For $s\in {\bf R}$ and $\sigma\in {\bf R}$, define the inhomogeneous weighted Sobolev space by
$$
H_{\sigma}^{s}({\bf R}^{N})=\{f\in S'({\bf R}^{N}):
\|\langle x\rangle^{\sigma}\langle i\nabla\rangle^{s}f\|_{L^{2}({\bf R}^{N})}<\infty\}
$$
and the homogeneous weighted Sobolev space by
$$\dot{H}_{\sigma}^{s}({\bf R}^{N})=\{f\in S'({\bf R}^{N}):
\|\langle x\rangle^{\sigma} |\nabla|^{s}f\|_{L^{2}({\bf R}^{N})}<\infty\},
$$
where $S'({\bf R}^{N})$ denotes the space of tempered distributions and $\langle x\rangle=(1+|x|^{2})^{\frac{1}{2}}$.
When $\sigma=0$, $H^{s}({\bf R}^{N})$ ($\dot{H}^{s}({\bf R}^{N})$) denotes the space $H_{0}^{s}({\bf R}^{N})$ ($\dot{H}_{0}^{s}({\bf R}^{N})$)
and when $s=0$, $L_{\sigma}^{2}({\bf R}^{N})$ denotes $H_{\sigma}^{0}({\bf R}^{N})$.

Given $p\geq 1$, let $p'$ be the conjugate of $p$, that is $\frac{1}{p}+\frac{1}{p'}=1$.

{\bf Acknowledgement}\ \ The first author is financially supported by the
China National Science Foundation (No.11301564, 11771469), the second author is
financially supported by the
China National Science Foundation ( No. 11771165 and
11571131), and the third author is financially supported by the
China National Science Foundation( No. 11771165).



\section{Strichartz type estimates associated with $H=\Delta^2+V$}\label{sec-2}
\setcounter{equation}{0}

We start in this section with recalling the Strichartz estimates of linear bi-harmonic Schr\"odinger equations with $V=0$.
 We say a pair $(q, r)$ is
{\it Schrodinger admissible}, or {\it S-admissible} for short, if $2\leq q, r\leq\infty$, $(q, r, N)\neq (2, \infty, 2)$ and
$$
\frac{2}{q}+\frac{N}{r}=\frac{N}{2}.
$$
Also, we use the terminology that a pair $(q, r)$ is {\it biharmonic admissible}, or {\it B-admissible} for short, if
$2\leq q, r\leq\infty$, $(q, r, N)\neq (2, \infty, 4)$ and
$$
\frac{4}{q}+\frac{N}{r}=\frac{N}{2}.
$$
We define the Strichatz norm by
$$
\|u\|_{S(L^{2}, I)}:=\sup_{(q,r):{\rm B-admissible}}\Vert u\Vert_{L^q(I,L^r)}
$$
and its dual norm by
$$
\|u\|_{S'(L^{2}, I)}:=\inf_{(q,r):{\rm B-admissible}}\Vert u\Vert_{L^{q'}(I,L^{r'})}
$$

The Strichartz estimates are stated as follows ( see e.g. \cite{Pau3} ):
\begin{lemma} \label{propo2.1}Let $u\in C(I, H^{-4}({\bf R}^{N}))$ be a solution of
 \begin{equation}\label{2.1}
 u(t)=e^{it\Delta^2}u_0+i\int_0^te^{i(t-s)\Delta^2}h(s)ds.
\end{equation}
Then we have
 \begin{equation}\label{2.2}
\Vert u\Vert_{S(L^{2}, I)}\leq C \ \Big(\Vert u_0\Vert_{L^2}+ \Vert
h\Vert_{S'(L^{2}, I)}\Big).
\end{equation}
More generally, for any S-admissible pairs $(q,r)$
and $(a,b)$, and any $s\geq0$,
 \begin{equation}\label{2.3}
\Vert |\nabla|^{s}u\Vert_{L^q(I,L^r)}\leq C\  \Big(\Vert
|\nabla|^{s-\frac2q}u_0\Vert_{L^2}+ \Vert
|\nabla|^{s-\frac2q-\frac2a}h\Vert_{L^{a'}(I,L^{b'})}\Big).
\end{equation}
\end{lemma}
Note that from Sobolev embedding inequality,  the estimate \eqref{2.3} implies the estimate \eqref{2.2}.
Thus a direct consequence of \eqref{2.3} and the  Sobolev's inequality is that,
  if $u\in C(I,H^{-4}({\bf{R}}^N))$ be a solution of \eqref{2.1}
 with $u_0\in\dot H^2$ and $\nabla h\in L^{2}(I,L^{\frac{2N}{N+2}})$,
 then  $u\in C(I,\dot H^{2}({\bf{R}}^N))$
 and for any B-admissible pairs $(q,r)$,
 \begin{equation}\label{2.4}
\Vert\Delta u\Vert_{L^q(I,L^r)}\leq C \ \Big(\Vert \Delta u_0\Vert_{L^2}+
\Vert \nabla h\Vert_{L^{2}(I,L^{\frac{2N}{N+2}})}\Big).
\end{equation}
A key feature of \eqref{2.4} is that the second derivative of $u$ is estimated using only
one derivative of the forcing term $h$. The same argument gives
 \begin{equation}\label{smooth}
\Vert|\nabla| u\Vert_{L^q(I,L^r)}\leq C\ \Big(\Vert |\nabla| u_0\Vert_{L^2}+
\Vert  h\Vert_{L^{2}(I,L^{\frac{2N}{N+2}})}\Big).
\end{equation}


If $u_0\in \dot{H}^{s}(\mathbb{R}^N)$, then we also can establish a $\dot{H}^{s}$-version to the Strichartz inequality \eqref{2.2}. More precisely, we introduce that a pair  $(q,r)$ is {\it $\dot{H}^{s}$-biharmomic admissible} and denote it by
 $(q,r)\in\Lambda_{s}$ if $0\leq s<2$ and
$$\frac{4}{q}+\frac{N}{r}=\frac{N}{2}-s,\ \ \ \frac{2N}{N-2s}\leq r\leq\frac{2N}{N-4}.$$
Correspondingly, we call the pair $(q',r')$
{\it dual~$\dot{H}^{s}$-biharmomic  admissible}, denoted by
$(q',r')\in\Lambda'_{s}$, if
 $(q,r)\in\Lambda_{-s}$ and $(q',r')$ is the conjugate exponent pair of $(q,r).$
In particular, {\it $(q,r)\in\Lambda_0$ is just a B-admissible pair, which is always denoted by $(q,r)\in\Lambda_ B$.}

We also define the exotic Strichartz norm by
$$
\|u\|_{S(\dot{H}^{s},I)}:=\sup_{(q,r)\in\Lambda_{s}}\|u\|_{L^q(I;L^r)},
$$
and its dual norm by
$$
\|u\|_{S'(\dot{H}^{-s},I)}:=\inf_{(q,r)\in\Lambda_{-s}}\|u\|_{L^{q'}(I;L^{r'})}.
$$
Now we can infer
the following  $\dot{H}^{s}$-Strichartz estimates on  $I=[0,T]$:
\begin{align}\label{2.5}
\|u\|_{S(\dot{H}^{s},I)}=\Big\|e^{it\Delta^2}u_0+i\int_0^te^{i(t-s)\Delta^2}h(\cdot,s)ds\Big\|_{S(\dot{H}^{s},I)}\leq C\ \Big(\|u_0\|_{\dot{H}^{s}}+\|h\|_{S'(\dot{H}^{-s},I)}\Big).
\end{align}
If the time interval $I$ is not specified, we take $I={\bf R}$.
We also refer to \cite{Fos, Kato, Keel, Pau3, Tag, Vile} for more discussion on the homogeneous and inhomogeneous type Strichartz estimates.

Next, we need to establish the Strichartz type estimates corresponding to \eqref{2.2}-\eqref{2.5} for solutions of
inhomogeneous linear biharmonic Schr\"{o}dinger equation with potential $V$:
\begin{equation}\label{2.6}
\left\{ \begin{aligned}
  iu_{t}+(\Delta^2+V)u+h=0,\;\;(t,x) \in {{\bf R}\times{\bf R}^{N}}, \\
  u(0, x)=u_{0}(x)\in H^{2}({\bf R}^{N}),
                          \end{aligned}\right.
                          \end{equation}
where $N\geq 5$. For the purpose, we now recall the following Local decay estimate, Jensen-Kato estimate and Strichartz type estimate established by Feng-Soffer-Yao \cite{FSY}.

\begin{lemma}\label{lemFSY}
Let $N\ge 5$, $V$ satisfy that $\langle x\rangle^{\beta}V(x)
\in L^{\infty}({\bf R}^{N})$ for some $\beta>N+4$.
And assume that operator $H=\Delta^2+V$ has no positive embedded
eigenvalues and 0 is a regular point for $H$. If $u$ be the solution of the initial value problem \eqref{2.6}, then the following estimates hold:
 \begin{equation}\label{2.4'}
\displaystyle\int_{{\bf R}}\big\Vert\langle x\rangle^{-\sigma} e^{-itH}P_{ac}u_{0}\big\Vert_{L^{2}({\bf R}^{N})}^{2}dt\leq C\ \|u_{0}\|_{L^{2}({\bf R}^{N})}^{2} \;\; {\text for} \;\;\sigma>1/2,
\end{equation}
 \begin{equation}\label{2.3'}
\Vert e^{-itH}P_{ac}u_{0}\Vert_{L_{-\sigma}^{2}({\bf R}^{N})}\leq C\  \langle t\rangle^{-N/4}\|u_{0}\|_{L_{\sigma}^{2}({\bf R}^{N})} \;\; {\text for} \;\;\sigma>N/2+2,
\end{equation}
\begin{equation}\label{2.2'}
\Vert P_{ac}u\Vert_{S(L^{2}, I)}\leq C\  \big(\Vert u_0\Vert_{L^2}+ \Vert
h\Vert_{S'(L^{2}, I)}\big),
\end{equation}
where $L^2_{\sigma}(\mathbb{R}^{d})$ is the weighted $L^2$-function space, $P_{ac}$ is the projection on the absolutely continuous spectrum of $H$, and $S(L^{2}, I), S(L^{2}, I)$ are the Strichartz norm and its dual norm in Proposition \ref{propo2.1}, respectively.
\end{lemma}


We say that a resonance occurs at zero for $H=\Delta^2+V$, provided there is a distributional solution $u$ of the equation
$(\Delta^2+V)u=0$ such that for any $s>4-N/2$, $u\in L^{2}_{-s}(\mathbf{R}^{N})\setminus L^{2}(\mathbf{R}^{N})$. We call that zero is a regular point of $H$, which means that 0 is neither eigenvalue nor resonance.  Moreover, when $N>8$, $H=\Delta^2+V$ has no zero resonance( see \cite[Remark 2.8]{FSY}).

Now, based on the local decay estimates and Kato-Jensen estimates of Lemma \ref{lemFSY} above, by putting the further repulsive condition on $V$  and restriction on dimension $N$, we will further establish $\dot{H}^s$-Strichatz estimates with potential and global smoothing estimate. For the purpose, we will use the following conditions:
\vskip0.3cm
 $(\mathfrak{C}_1)$: \ {\it \ $V$ is a real $C^1$-function of ${\bf R}^{N}$ satisfying that $x\cdot \nabla V\leq 0$ and $|V(x)|+|\nabla V|\le C (1+ |x|)^{-\beta}$ for some $\beta>N+4$.}

 $(\mathfrak{C}_2)$: {\it \ \ $0$ is not a resonance of $H=\Delta^2+V$.  }
\vskip0.3cm

It was well-known from Virial's argument that the repulsive condition $x\cdot \nabla V\leq 0$ makes that the operator $H=\Delta^2+V$ has no any eigenvalue in $\mathbf{R}$ ( i.e. $\sigma_p(H)=\emptyset$ ), e.g. see Reed and Simon \cite[Theorem XIII.59]{RS2} for Schr\"odinger operator, and Feng \cite{feng} for general operators $P(D)+V$.  {\it Hence it follows from the assumptions on $V$ and dimension $N$ in Theorem \ref{th1.2}
that the operator $H=\Delta^2+V$ has no any eigenvalues ( i.e . $\sigma(H)=\sigma_{ac}(H)=[0,\infty)$ ), and zero is not a resonance. }Thus, $P_{ac}$ is an identity operator, and the above estimates \eqref{2.4'}-\eqref{2.2'}  still hold true for $u$ in place of $P_{ac}u$

\vskip0.3cm

\begin{proposition}\label{lem2.1}
Let $0\le s<2$, $N>4+2s$, $H=\Delta^2+V$ and $V$ satisfy the conditions $(\mathfrak{C}_1)$-$(\mathfrak{C}_2)$ .  If $u$ be the solution of the initial value problem \eqref{2.6}, then we have
\begin{align}\label{2.7}
\|u\|_{S(\dot{H}^{s},I)}=\Big\|e^{itH}u_0+i\int_0^te^{i(t-s)H}h(\cdot,s)ds\Big\|_{S(\dot{H}^{s},I)}\leq C\ \Big(\|u_0\|_{\dot{H}^{s}}+\|h\|_{S'(\dot{H}^{-s},I)}\Big).
\end{align}
where $S(\dot{H}^{s},I)$ and $S'(\dot{H}^{s},I)$ are the exotic Strichartz norm and its dual norm in \eqref{2.5}.
\end{proposition}

\begin{proof} The solution $u$ to the problem \eqref{2.6} can be expressed as
\begin{align}\label{2.8'}
u(t,x)=e^{itH}u_{0}+i\displaystyle\int_{0}^{t}e^{i(t-s)H}h(s)ds\doteq u_{1}(t,x)+u_{2}(t,x),
\end{align}
where $u_{1}$ and $u_{2}$ may also be expressed respectively as
\begin{align}\label{2.9'}
u_{1}(t,x)=e^{it\Delta^2}u_{0}+i\displaystyle\int_{0}^{t}e^{i(t-s)\Delta^2}Vu_{1}(s)ds
\end{align}
and
\begin{align}\label{2.10'}
u_{2}(t,x)=i\displaystyle\int_{0}^{t}e^{i(t-s)\Delta^2}Vu_{2}(s)ds
+i\displaystyle\int_{0}^{t}e^{i(t-s)\Delta^2}h(s)ds.
\end{align}

We first use the Jensen-Kato type decay estimate \eqref{2.3'} to control $u_{1}(t,x)$.

Using \eqref{2.5} and H\"{o}lder inequality successively yields that
\begin{align}\label{2.11'}
\Vert u_{1}\Vert_{S(\dot{H}^{s}, I)}
\leq & C\Big(\Vert u_0\Vert_{\dot{H}^{s}}+\Vert Vu_{1}\Vert_{L_{I}^{2}L_{x}^{\frac{2N}{N+4-2s}}}\Big)\nonumber\\
\leq & C\Big(\Vert u_0\Vert_{\dot{H}^{s}}+\Vert \langle x\rangle^{\sigma}V\Vert_{L^{\frac{N}{2-s}}}\ \Vert\langle x\rangle^{-\sigma} u_{1}\Vert_{L_{I}^{2}L_{x}^{2}}\Big).
\end{align}
Now we aim to show that
\begin{align}\label{2.12'}
\Vert\langle x\rangle^{-\sigma} u_{1}\Vert_{L_{I}^{2}L_{x}^{2}}\lesssim \Vert u_0\Vert_{\dot{H}^{s}}.
\end{align}
Consider the Cauchy problem
\begin{equation}\label{2.13'}
\left\{ \begin{aligned}
  i\phi_{t}=-\Delta^2\phi=-H\phi+V\phi, \\
  u(0, x)=u_{0}(x)\in\dot{H}^{s}(\mathbf{R}^N).
                          \end{aligned}\right.
                          \end{equation}
Then Duhamel formula for the solution $\phi$ gives
\begin{align}\label{2.14'}
\phi(t)=e^{it\Delta^2}u_{0}=e^{itH}u_{0}-i\displaystyle\int_{0}^{t}e^{i(t-s)H}V\phi(s)ds=u_{1}(t)-i\displaystyle\int_{0}^{t}e^{i(t-s)H}V\phi(s)ds,
\end{align}
Hence
\begin{align}\label{2.15'}
u_{1}(t)=\phi(t)+i\displaystyle\int_{0}^{t}e^{i(t-s)H}V\phi(s)ds=e^{it\Delta^2}u_{0}+i\displaystyle\int_{0}^{t}e^{i(t-s)H}V\phi(s)ds.
\end{align}
By H\"{o}lder inequality, \eqref{2.2} and Sobolev embedding estimates, we have
\begin{align}\label{2.16'}
\Vert\langle x\rangle^{-\sigma} \phi\Vert_{L_{I}^{2}L_{x}^{2}}=&\Vert\langle x\rangle^{-\sigma} e^{itH_{0}}u_{0}\Vert_{L_{I}^{2}L_{x}^{2}}\nonumber\\
\le  & \|\langle x\rangle^{-\sigma}\|_{L_{x}^{\frac{N}{2+s}}}\|e^{itH_{0}}u_{0}\|_{L_{I}^{2}L_{x}^{\frac{2N}{N-4-2s}}}\nonumber\\
\lesssim & \Vert u_0\Vert_{\dot{H}^{s}}.
\end{align}
The second term can be estimated by the Jensen-Kato type decay estimates \eqref{2.3'} and Young inequality:
\begin{align}\label{2.17'}
\Big\|\langle x\rangle^{-\sigma}\displaystyle\int_{0}^{t}e^{i(t-s)H}V\phi(s)ds\Big\|_{L_{I}^{2}L_{x}^{2}}
\le&\Big\|\displaystyle\int_{0}^{t}\|\langle x\rangle^{-\sigma}e^{i(t-s)H}V\phi(s)\|_{L_{x}^{2}}ds\Big\|_{L_{I}^{2}}\nonumber\\
\lesssim &\Big\|\displaystyle\int_{0}^{t}\langle t-s\rangle^{-\frac{N}{4}}\|\langle x\rangle^{\sigma} V\phi(s)\|_{L_{x}^{2}}ds\Big\|_{L_{I}^{2}}\nonumber\\
\lesssim &\|\langle x\rangle^{\sigma} V\phi(s)\|_{L_{I}^{2}L_{x}^{2}}\lesssim \|\langle x\rangle^{2\sigma} V\|_{L_{x}^{\infty}}\|\langle x\rangle^{-\sigma}\phi\|_{L_{I}^{2}L_{x}^{2}}\nonumber\\
\lesssim &\Vert u_0\Vert_{\dot{H}^{s}}.
\end{align}
Putting \eqref{2.16'} and \eqref{2.17'} together gives the desired \eqref{2.12'}.

Next, by using the same argument we will show that
\begin{align}\label{2.18'}
\Vert u_{2}\Vert_{S(\dot{H}^{s}, I)}\lesssim  \Vert h\Vert_{S'(\dot{H}^{-s}, I)}.
\end{align}
Indeed, using \eqref{2.5} gives that
\begin{align}\label{2.19'}
\Vert u_{2}\Vert_{S(\dot{H}^{s}, I)}
\leq & C\ \Big(\Vert h\Vert_{S'(\dot{H}^{-s}, I)})+\Vert Vu_{2}\Vert_{L_{I}^{2}L_{x}^{\frac{2N}{N+4-2s}}}\Big)\nonumber\\
\leq & C\ \Big(\Vert h\Vert_{S'(\dot{H}^{-s}, I)})+\Vert \langle x\rangle^{\sigma}V\Vert_{L^{\frac{N}{2-s}}}\Vert\langle x\rangle^{-\sigma} u_{2}\Vert_{L_{I}^{2}L_{x}^{2}}\Big).
\end{align}
Now we aim to show that
\begin{align}\label{2.20'}
\Vert\langle x\rangle^{-\sigma} u_{2}\Vert_{L_{I}^{2}L_{x}^{2}}\lesssim \Vert h\Vert_{S'(\dot{H}^{-s}, I)}.
\end{align}
Consider the Cauchy problem
\begin{equation}\label{2.21'}
\left\{ \begin{aligned}
  i\phi_{t}=-\Delta^2\phi-h(t)=-H\phi+V\phi-h(t), \\
  u(0, x)=0.\quad \quad \quad\quad \quad \;\;\quad\quad \quad\quad
                          \end{aligned}\right.
                          \end{equation}
Then Duhamel formula for the solution $\phi$ reads
\begin{align}\label{2.22'}
\phi(t)=i\displaystyle\int_{0}^{t}e^{i(t-s)H_{0}}h(s)ds=-i\displaystyle\int_{0}^{t}e^{i(t-s)H}V\phi(s)ds+i\displaystyle\int_{0}^{t}e^{i(t-s)H}h(s)ds=-
i\displaystyle\int_{0}^{t}e^{i(t-s)H}V\phi(s)ds+u_{2}(t),
\end{align}
Hence
\begin{align}\label{2.23'}
u_{2}(t)=\phi(t)+i\displaystyle\int_{0}^{t}e^{i(t-s)H}V\phi(s)ds=i\displaystyle\int_{0}^{t}e^{i(t-s)H_{0}}h(s)ds+i\displaystyle\int_{0}^{t}e^{i(t-s)H}V\phi(s)ds.
\end{align}
By H\"{o}lder inequality and \eqref{2.5}, we have
\begin{align}\label{2.24'}
\Vert\langle x\rangle^{-\sigma} \phi\Vert_{L_{I}^{2}L_{x}^{2}}=&\Big\Vert\langle x\rangle^{-\sigma} \displaystyle\int_{0}^{t}e^{i(t-s)H_{0}}h(s)ds\Big\Vert_{L_{I}^{2}L_{x}^{2}}\nonumber\\
\lesssim & \big\|\langle x\rangle^{-\sigma}\big\|_{L_{x}^{\frac{N}{2+s}}}\Big\|\displaystyle\int_{0}^{t}e^{i(t-s)H_{0}}h(s)ds\Big\|_{L_{I}^{2}L_{x}^{\frac{2N}{N-4-2s}}}\nonumber\\
\lesssim & \Vert h\Vert_{S'(\dot{H}^{-s}, I)}.
\end{align}
The other term can be estimated by the Jensen-Kato type decay estimates \eqref{2.3'} and Young inequality:
\begin{align}\label{2.25'}
\Big\|\langle x\rangle^{-\sigma}\displaystyle\int_{0}^{t}e^{i(t-s)H}V\phi(s)ds\Big\|_{L_{I}^{2}L_{x}^{2}}
\lesssim &\Big\|\displaystyle\int_{0}^{t}\|\langle x\rangle^{-\sigma}e^{i(t-s)H}V\phi(s)\|_{L_{x}^{2}}ds\Big\|_{L_{I}^{2}}\nonumber\\
\lesssim &\Big\|\displaystyle\int_{0}^{t}\langle t-s\rangle^{-\frac{N}{4}}\|\langle x\rangle^{\sigma} V\phi(s)\|_{L_{x}^{2}}ds\Big\|_{L_{I}^{2}}\nonumber\\
\lesssim &\|\langle x\rangle^{\sigma} V\phi(s)\|_{L_{I}^{2}L_{x}^{2}}\lesssim \|\langle x\rangle^{2\sigma} V\|_{L_{x}^{\infty}}\|\langle x\rangle^{-\sigma}\phi\|_{L_{I}^{2}L_{x}^{2}}\nonumber\\
\lesssim & \Vert h\Vert_{S'(\dot{H}^{-s}, I)}.
\end{align}
Putting \eqref{2.24'} and \eqref{2.25'} together gives \eqref{2.20'}.
Collecting \eqref{2.11'}, \eqref{2.12'} and \eqref{2.18'} completes the proof of \eqref{2.7}.
\end{proof}

\begin{remark}\label{rem2.2}
Note that the constant $C $ in the estimate \eqref{2.7} is
dependent of the scaling $V(x)\mapsto V_{r}=\frac{1}{r^{4}}V(\frac{x}{r})$, which is fundamental to the
linear profile decomposition Proposition \ref{pro5.3}. Indeed, Let $C=C(V)>0$ be the sharp constant for \eqref{2.7}.
Then for any $(q, r)\in \Lambda_{s}$ and $(\tilde{q}, \tilde{r})\in \Lambda_{-s}$, we have
\begin{align}\label{2.25''}
\Big\|& e^{it(\Delta^2+V_{r})}f\Big(\frac{x}{r}\Big)\Big\|_{L_{t}^{q}L_{x}^{r}}+
\Big\|\displaystyle\int_{0}^{t}e^{i(t-s)(\Delta^2+V_{r})}h\Big(\frac{s}{r^{4}},\frac{x}{r}\Big)ds\Big\|_{L_{t}^{q}L_{x}^{r}}\nonumber\\
=&\Big\|e^{i\frac{t}{r^{4}}(\Delta^2+V)}f\Big(\frac{x}{r}\Big)\Big\|_{L_{t}^{q}L_{x}^{r}}
+\Big\|\displaystyle\int_{0}^{t}e^{i\frac{t-s}{r^{4}}(\Delta^2+V_{r})}h\Big(\frac{s}{r^{4}},\frac{x}{r}\Big)ds\Big\|_{L_{t}^{q}L_{x}^{r}}\nonumber\\
=& r^{\frac{4}{q}+\frac{N}{r}}\Big\|e^{itH}f\Big\|_{L_{t}^{q}L_{x}^{r}}+r^{\frac{4}{q}+\frac{N}{r}+4}\Big\|\displaystyle\int_{0}^{t}e^{i(t-s)H}h(s)ds\Big\|_{L_{t}^{q}L_{x}^{r}}\nonumber\\
\leq & C(V)\Big(\ r^{\frac{4}{q}+\frac{N}{r}}\|f\|_{\dot{H}^{s}}+r^{\frac{4}{q}+\frac{N}{r}+4}\|h\|_{L_{t}^{\tilde{q}'}L_{x}^{\tilde{r}'}}\Big)\nonumber\\
=& C(V)\Big(\ r^{\frac{4}{q}+\frac{N}{r}-\frac{N}{2}+s}\Big\|f\Big(\frac{x}{r}\Big)\Big\|_{\dot{H}^{s}}
+r^{\frac{4}{q}+\frac{N}{r}+4-\frac{4}{\tilde{q}'}-\frac{N}{\tilde{r}'}}\Big\|h\Big(\frac{t}{r^{4}},\frac{x}{r}\Big)\Big\|_{L_{t}^{\tilde{q}'}L_{x}^{\tilde{r}'}}\Big)\nonumber\\
=&C(V)\Big(\ \Big\|f\Big(\frac{x}{r}\Big)\Big\|_{\dot{H}^{s}}+\Big\|h\Big(\frac{t}{r^{4}},\frac{x}{r}\Big)\Big\|_{L_{t}^{\tilde{q}'}L_{x}^{\tilde{r}'}}\Big).
\end{align}
As $f$ are arbitrarily chosen, we have $C(V_{r})\le C(V)$ for all $r>0$. On the other hand, since $V(x)=r^4V_r(rx)$ for any $r>0$, by symmetry we can conclude that
$C(V_{r})= C(V)$ for all $r>0$.
\end{remark}

\begin{proposition}\label{lem2.3}
Let $N>6$ and $V$ satisfy the conditions $(\mathfrak{C}_1)$-$(\mathfrak{C}_2)$.  Then for any given B-admissible pair $(q, r)$,
the solution $u$ to \eqref{2.6} satisfies the inequality
\begin{align}\label{2.11}
\Vert\Delta u\Vert_{L^q(I,L^r)}\leq C (\Vert u_0\Vert_{H^2}+
\Vert \langle\nabla\rangle h\Vert_{L^{2}(I,L^{\frac{2N}{N+2}})}).
\end{align}
\end{proposition}

\begin{proof}
 The solution $u$ to \eqref{2.6} can be expressed as
\begin{align}\label{2.12}
u(t,x)=e^{it\Delta^2}u_{0}+i\displaystyle\int_{0}^{t}e^{i(t-s)\Delta^2}Vu(s)ds
+i\displaystyle\int_{0}^{t}e^{i(t-s)\Delta^2}h(s)ds.
\end{align}
Using \eqref{2.4} yields
\begin{align}\label{2.13}
\|\Delta u\|_{L_{I}^{q}L_{x}^{r}}\lesssim \|\Delta u_{0}\|_{L^{2}}
+\|\nabla h\|_{L_{I}^{2}L_{x}^{\frac{2N}{N+2}}}
+\|\nabla (Vu)\|_{L_{I}^{2}L_{x}^{\frac{2N}{N+2}}},
\end{align}
where
\begin{align}\label{2.14}
\|\nabla (Vu)\|_{L_{I}^{2}L_{x}^{\frac{2N}{N+2}}}
&\leq \|V(\nabla u)\|_{L_{I}^{2}L_{x}^{\frac{2N}{N+2}}}+\|(\nabla V)u\|_{L_{I}^{2}L_{x}^{\frac{2N}{N+2}}}
\nonumber\\
&\lesssim \|\langle x\rangle^{\sigma}\nabla V\|_{L_{x}^{N}}\|\langle x\rangle^{-\sigma}u\|_{L_{I,x}^{2}}
+\| V\|_{L_{x}^{\frac{N}{3}}}\|\nabla u\|_{L_{I}^{2}L_{x}^{\frac{2N}{N-4}}}.
\end{align}
Since
\begin{align}\label{2.15}
\|\nabla u\|_{L_{t}^{2}L_{x}^{\frac{2N}{N-4}}}
\lesssim & \|\nabla u_{0}\|_{L_{x}^{2}}
+\|h\|_{L_{I}^{2}L_{x}^{\frac{2N}{N+2}}}+\|Vu\|_{L_{I}^{2}L_{x}^{\frac{2N}{N+2}}}\nonumber\\
\lesssim & \|\nabla u_{0}\|_{L_{x}^{2}}
+\|h\|_{L_{I}^{2}L_{x}^{\frac{2N}{N+2}}}+\|\langle x\rangle^{\sigma} V\|_{L_{x}^{N}}\|\langle x\rangle^{-\sigma}u\|_{L_{I,x}^{2}},
\end{align}
where we have used the estimate \eqref{smooth} in the first inequality.
It follows from \eqref{2.13}-\eqref{2.15} that it suffices to control $\|\langle x\rangle^{-\sigma}u\|_{L_{I,x}^{2}}$.
We note that using \eqref{2.12'} and \eqref{2.20'} with $s=1$ in Proposition \ref{lem2.1} yields that
\begin{align}\label{weighted}
\|\langle x\rangle^{-\sigma}u\|_{L_{I,x}^{2}}\lesssim  \|\nabla u_{0}\|_{L_{x}^{2}}
+\|h\|_{L_{I}^{2}L_{x}^{\frac{2N}{N+2}}}.
\end{align}
Thus putting \eqref{2.13}-\eqref{weighted} together gives \eqref{2.11}.
\end{proof}

By the estimate \eqref{2.3}, when take $s=0$ and $q=2$, we have
$$\big\||\nabla|e^{it\Delta^2}u_0\big\|_{L^2_tL_x^{\frac{2N}{N-2}}}\lesssim \|u_0\|_{L^2_x}, $$
which by a dual argument deduces the following smoothing estimate:
\begin{align}\label{2.16}
\Big\|\Delta\displaystyle\int_{{\bf R}}e^{-is\Delta^2}h(s)ds\Big\|_{L_{x}^{2}}
\lesssim \big\||\nabla|h\big\|_{L_{t}^{2}L_{x}^{\frac{2N}{N+2}}}.
\end{align}
Similarly, to get the scattering result in our paper, we need the following estimate with $H$ in place of $\Delta^2$ inside the integral, which can be also used
to establish the linear scattering ( see Proposition \ref{lem2.7})
as follows.

\begin{proposition}\label{lem2.4}
Let $N>4$, $H_0=\Delta^2$, $H=H_0+V$ and $V$ satisfy the conditions $(\mathfrak{C}_1)$-$(\mathfrak{C}_2)$, then we have the smoothness estimate
\begin{align}\label{2.17}
\Big\|H_{0}^{\frac{1}{2}}\displaystyle\int_{{\bf R}}e^{-isH}h(s)ds\Big\|_{L_{x}^{2}}
\lesssim \big\||\nabla|h\big\|_{L_{t}^{2}L_{x}^{\frac{2N}{N+2}}}.
\end{align}
\end{proposition}

To get \eqref{2.17}, we need to show that the fractional power associated with $H$ is bounded by the one associated with $H_{0}$ in $L^r$-norm, because it compensate the non-commutativity between $H_{0}^{\frac{1}{4}}$ and $e^{itH}$.

\begin{lemma}\label{acta}
Let $N>4$, $0\le V\in L^{\frac{N}{4}}({\bf R}^{N})$, $H_0=\Delta^2$ and $H=H_0+V$. Then for $s\in [0, 4]$ and $$\frac{2N}{N+4}<r<\min\Big\{\frac{8N}{4(N-4)-(N-12)s},\frac{2N}{N-4}\Big\},$$
\begin{equation}\label{acta1}
\|H^{\frac{s}{4}}f\|_{L^{r}}\leq C_{r}\ \|H_{0}^{\frac{s}{4}}f\|_{L^{r}}
\sim \||\nabla|^{s}f\|_{L^{r}}.
\end{equation}
In particular, taking $s=1$ and $ \frac{2N}{N+2}\le r\le \frac{2N}{N-2}$, we have
\begin{equation}\label{acta1}
\|H^{\frac{1}{4}}f\|_{L^{r}}\leq C_N\ \|H_{0}^{\frac{1}{4}}f\|_{L^{r}}
\sim \||\nabla|f\|_{L^{r}}.
\end{equation}
\end{lemma}

\begin{proof} By Fourier transform we can define the power of $H_0$ as follows
$$\widehat{H_0^{z}f}(\xi)=|\xi|^{4z}\widehat{f}(\xi),\  z\in \mathbf{C}$$
where $\widehat{f}$ denotes the Fourier transform of $f$.  When $z=iy$, it was well-known from Mihilin's multiplier theorem that
the imaginary power operator $H_{0}^{-iy}$ are  bounded on
$L^{q}$ for all $1<q<\infty$.

For the nonnegative self-adjoint operator $H=\Delta^2+V$, $H^{z}$ can be defined by the functional calculus
\begin{equation}\label{acta3}
H^{z}=\displaystyle\int_{0}^{\infty}\lambda^{z}dE_{H}(\lambda).
\end{equation}
Since the fourth order Schr\"odinger semigroup $e^{-tH}$ satisfies the following $(p,q)$ off-diagonal estimates ( See \cite[Section 2]{DDY} ):
\begin{eqnarray}\label{gges}
\big\|\chi_{B(x,t^{1/4})}e^{-tH}\chi_{B(y,t^{1/4})}\big\|_{L^{p}\rightarrow L^{q}}\leq Ct^{\frac{N}{4}(\frac{1}{q}-\frac{1}{p})}e^{-c\frac{|x-y|^{4/3}}{t^{1/3}}},
\end{eqnarray}
for $\frac{2N}{N+4}\le p\le q\le \frac{2N}{N-4}$ and $N>4$, where $B(x,t^{1/4})$ is the ball centered at $x$ with radius $t^{1/4}$ and $\chi_{B(x,t^{1/4})}$ is the characteristic function of $B(x,t^{1/4})$, it follows from \cite[Theorem 1.2]{Blu} that the imaginary power operator $H^{-iy}$ are  bounded on
$L^{q}$ for all $\frac{2N}{N+4}<q<\frac{2N}{N-4}$.

Define a family of operators $T_z$ as follows:
\begin{equation}\label{acta2}
T_{z}=H^{z}H_{0}^{-z},\quad z=x+iy\in \mathbf{C},\  0\le x\le 1.
\end{equation}
Obviously, for $z=iy$ and all $\frac{2N}{N+4}<q<\frac{2N}{N-4}$,
\begin{equation}\label{acta4}
\|T_{iy}\|_{L^{q}\rightarrow L^{q}}\lesssim (1+|y|)^{2\alpha}, \ \alpha\ge[N/2]+1.
\end{equation}
Now turn to $z=1+iy$. For $1<p<\frac{N}{4}$, by using H\"{o}lder inequality and
Sobolev embedding theorem we get that
\begin{equation}\label{acta5}
  \begin{split}
      \|Hf\|_{L^{p}}\leq&\,\|H_{0}f\|_{L^{p}}+\|Vf\|_{L^{p}}\\
                    \leq&\, \|H_{0}f\|_{L^{p}}+\|V\|_{L^{\frac{N}{4}}}\|H_{0}f\|_{L^{p}}\\
                    \lesssim&\, \|H_{0}f\|_{L^{p}},
  \end{split}
\end{equation}
which implies that for all  $\frac{2N}{N+4}<q<\min\{\frac{2N}{N-4},\ \frac{N}{4} \}$
\begin{equation}\label{acta6}
\|T_{1+iy}\|_{L^{q}\rightarrow L^{q}}=\|H^{iy}T_{1}H_0^{-iy}\|_{L^{q}\rightarrow L^{q}}\lesssim (1+|y|)^{2\alpha}, \ \alpha\ge[N/2]+1.
\end{equation}
By collecting \eqref{acta4} and \eqref{acta6}, it follows from Stein-Weiss interpolation that
real number $\theta\in [0, 1]$,
\begin{equation}\label{acta7}
\|T_{\theta}\|_{L^{r}\rightarrow L^{r}}\lesssim 1,
\end{equation}
for all $$\frac{2N}{N+4}<r<\min\Big\{\frac{2N}{(N-4)-(N-12)\theta},\frac{2N}{N-4}\Big\}.$$
Let $s=4\theta$, then for $s\in [0, 4]$, we obtain the following desired estimates
\begin{equation}\label{acta8}
\|H^{\frac{s}{4}}f\|_{L^{r}}\leq C_{r}\|H_{0}^{\frac{s}{4}}f\|_{L^{r}}
\sim \||\nabla|^{s}f\|_{L^{r}}.
\end{equation}
In particular, when $s=1$  and $ \frac{2N}{N+2}\le r\le \frac{2N}{N-2}$,  we have
$$\|H^{\frac{1}{4}}f\|_{L^{r}}\leq C_N\ \|H_{0}^{\frac{1}{4}}f\|_{L^{r}}
\sim \||\nabla|f\|_{L^{r}}.$$
\end{proof}

\begin{remark}\label{Gaussian}
In the $(p.q)$-off diagonal estimate \eqref{gges} of $e^{-t(\Delta^2+V)}$, if $(p,q)=(1,\infty)$, then the estimate  \eqref{gges} is equivalent to the following Gaussian kernel estimate:
 \begin{eqnarray}\label{g-kernel}
|e^{-tH}(x,y)|\leq Ct^{-\frac{N}{4}}e^{-c\frac{|x-y|^{4/3}}{t^{1/3}}},\ \ t>0
\end{eqnarray}
for some constants $C, c>0$. It was well-known that the semigroup $e^{-t\Delta^2}$ satisfies the point-wise estimate \eqref{g-kernel}. So it would be desirable to obtain the Gaussian kernel estimate \eqref{g-kernel} for $e^{-t(\Delta^2+V)}$ in the case $V\ge 0$, which actually implies the $(p.q)$-off diagonal estimate \eqref{gges} for all $1\le p\le q\le\infty$.   However, we remark that the higher order semigroup $e^{-t\Delta^2}$ is not positivity-preserving
and also not a contractive one on $L^{p}(\mathbf{R}^{N})$
$(p\neq2)$ ( see e.g.\cite{L-M} ), which becomes
 difficult to use famous Trotter formula to establish the pointwise the estimate \eqref{g-kernel} for $e^{-t(\Delta^2+V)}$. As for more studies about the higher order kernel estimates, one can see \cite{D-D-Y4} and  references therein.
\end{remark}

\begin{proof}[Proof of Proposition \ref{lem2.4}]
 By Remark \ref{rem1.3}, H\"{o}lder inequality and \eqref{acta1}, we have
\begin{align}\label{2.23}
&\Big\|H_{0}^{\frac{1}{2}}\displaystyle\int_{{\bf R}}e^{-isH}h(s)ds\Big\|_{L_{x}^{2}}\leq\Big\|H^{\frac{1}{2}}\displaystyle\int_{{\bf R}}e^{-isH}h(s)ds\Big\|_{L_{x}^{2}}\nonumber\\
&= \sup_{\|g\|_{L_{x}^{2}}\leq 1}\Big\langle\displaystyle\int_{{\bf R}}e^{-isH}H^{\frac{1}{2}}h(s)ds, g(x)\Big\rangle\nonumber\\
&= \sup_{\|g\|_{L_{x}^{2}}\leq 1}\displaystyle\int_{{\bf R}}\Big\langle H^{\frac{1}{4}}h(s), e^{isH}H^{\frac{1}{4}}g\Big\rangle ds\nonumber\\
&\lesssim \|H^{\frac{1}{4}}h\|_{L_{t}^{2}L_{x}^{\frac{2N}{N+2}}}
\sup_{\|g\|_{L_{x}^{2}}\leq 1}\|e^{isH}H^{\frac{1}{4}}g\|_{L_{t}^{2}L_{x}^{\frac{2N}{N-2}}}\nonumber\\
&\lesssim \big\||\nabla|h\big\|_{L_{t}^{2}L_{x}^{\frac{2N}{N+2}}}
\sup_{\|g\|_{L_{x}^{2}}\leq 1}\big\||\nabla|e^{isH}g\big\|_{L_{t}^{2}L_{x}^{\frac{2N}{N-2}}}.
\end{align}
Now it suffices to prove that
\begin{align}\label{2.24}
\big\||\nabla|e^{isH}g\big\|_{L_{t}^{2}L_{x}^{\frac{2N}{N-2}}}\lesssim \|g\|_{L_{x}^{2}}.
\end{align}
Indeed, let $u(t,x)=e^{itH}g$, then it can be  expressed as
$$
u(t)=e^{it\Delta^2}g+i\displaystyle\int_{0}^{t}e^{i(t-s)\Delta^2}(Vu(s))ds,
$$
 and using \eqref{2.3} with $s=1$ and $(q, r)=(2, \frac{2N}{N-2})$, Sobolev embedding,
H\"{o}lder inequality and Sobolev embedding again leads to
\begin{align}\label{2.25}
\big\||\nabla|u\big\|_{L_{t}^{2}L_{x}^{\frac{2N}{N-2}}}
&\lesssim \big\||\nabla|e^{isH_{0}}g\big\|_{L_{t}^{2}L_{x}^{\frac{2N}{N-2}}}
+\Big\||\nabla|\displaystyle\int_{0}^{t} e^{i(t-s)H_{0}}Vu(s)ds\Big\|_{L_{t}^{2}L_{x}^{\frac{2N}{N-2}}}\nonumber\\
&\lesssim  \big\|g\big\|_{L_{x}^{2}}+\big\||\nabla|^{-1}(Vu)\big\|_{L_{t}^{2}L_{x}^{\frac{2N}{N+2}}}\nonumber\\
&\lesssim  \|g\|_{L_{x}^{2}}+\|Vu\|_{L_{t}^{2}L_{x}^{\frac{2N}{N+4}}}\nonumber\\
&\lesssim  \|g\|_{L_{x}^{2}}
+\|\langle x\rangle^{\sigma}V\|_{L_{x}^{\frac{N}{2}}}\|\langle x\rangle^{-\sigma}u\|_{L_{t}^{2}L_{x}^{2}}\nonumber\\
&\lesssim  \|g\|_{L_{x}^{2}},
\end{align}
where in the last step  we have used local decay estimate \eqref{2.4'}.
\end{proof}


\begin{remark}\label{remark scaling}
Just as what we said in Remark \ref{rem2.2}, the constant $c$ for the estimate \eqref{2.17} is
dependent of the scaling $V(x)\mapsto V_{r}=\frac{1}{r^{4}}V(\frac{x}{r})$, which is also another important fact for the
linear profile decomposition Proposition \ref{pro5.3}. Indeed, Let $c=c(V)>0$ be the sharp constant for \eqref{2.17}.
Then we have
\begin{align}\label{scaling}
\Big\|\Delta \displaystyle\int_{{\bf R}}e^{-is(H_{0}+V_{r})}h\Big(\frac{s}{r^{4}},\frac{x}{r}\Big)\Big\|_{L_{x}^{2}}
=&\Big\|\Delta \Big(\displaystyle\int_{{\bf R}}e^{-i\frac{s}{r^{4}}(H_{0}+V)}h\Big(\frac{s}{r^{4}},\frac{x}{r}\Big)ds\Big)\Big\|_{L_{x}^{2}}
= r^{2+\frac{N}{2}}\Big\|\Delta \displaystyle\int_{{\bf R}}e^{-isH}h(s)ds\Big\|_{L_{x}^{2}}\nonumber\\
\leq c(V)r^{2+\frac{N}{2}}\|\nabla h\|_{L_{t}^{2}L_{x}^{\frac{2N}{N+2}}}
=& c(V)\Big\|\nabla \Big(h\Big(\frac{t}{r^{4}}, \frac{x}{r}\Big)\Big)\Big\|_{L_{t}^{2}L_{x}^{\frac{2N}{N+2}}}.
\end{align}
As $r$ and $h$ are arbitrarily chosen, we find that $c(V_{r})=c(V)$ for all $r>0$.
\end{remark}

As a simple application of Proposition \ref{lem2.4}, we shall establish the following linear scattering.

\begin{proposition}\label{lem2.7}
Let $V$ satisfy the conditions $(\mathfrak{C}_1)$-$(\mathfrak{C}_2)$ and $N>8$. Then

 (i) For any given $\phi\in L^{2}({\bf R}^{N})$, there exist $\phi^{\pm}$ such that
\begin{align}\label{2.26}
\lim_{t\rightarrow\pm\infty}\|e^{itH_{0}}\phi-e^{itH}\phi^{\pm}\|_{L^2({\bf R}^{N})}=0.
\end{align}

(ii) For any given $\phi\in H^{2}({\bf R}^{N})$, there exist $\phi^{\pm}$ such that
\begin{align}\label{2.27}
\lim_{t\rightarrow\pm\infty}\|e^{itH_{0}}\phi-e^{itH}\phi^{\pm}\|_{H^2({\bf R}^{N})}=0.
\end{align}
\end{proposition}

\begin{proof}
 (i) Let $u(t)=e^{itH_{0}}\phi$, then it can be also expressed as
\begin{align}\label{2.28}
u(t)=e^{itH}\phi-i\displaystyle\int_{0}^{t}e^{i(t-s)H}(Vu(s))ds.
\end{align}
Applying Strichartz estimates \eqref{2.7} with $s=0$ yields that for any $t_{1}, t_{2}\in {\bf R}$,
\begin{align}\label{2.29}
&\| e^{-it_{1}H}e^{it_{1}H_{0}}\phi-e^{-it_{2}H}e^{it_{2}H_{0}}\phi\|_{L^{2}}
=\|e^{-it_{1}H}u(t_{1})-e^{-it_{2}H}u(t_{2})\|_{L^{2}}\nonumber\\
&=\Big\|\displaystyle\int_{t_1}^{t_2}e^{-isH}(Vu(s))ds\Big\|_{L^{2}}
\lesssim  \|Vu(t)\|_{L_{[t_{1}, t_{2}]}^{2}L_{x}^{\frac{2N}{N+4}}}
\lesssim \|V\|_{L_{x}^{\frac{N}{4}}}\|u\|_{L_{[t_{1}, t_{2}]}^{2}L_{x}^{\frac{2N}{N-4}}}.
\end{align}
Using Strichartz estimates \eqref{2.2}, it is easy to find that
\begin{align}\label{2.30}
\|u\|_{L_{[t_{1}, t_{2}]}^{2}L_{x}^{\frac{2N}{N-4}}}\rightarrow 0
\end{align}
as $t_{1}, t_{2}\rightarrow \pm\infty$, which implies that the limit of $e^{-itH}e^{itH_{0}}\phi$ exists in $L^{2}$ as
$t$ tends to $\pm\infty$, and it is namely $\phi^{\pm}$ we need to find.

In fact, repeating the process of \eqref{2.29} yields that
\begin{align}\label{2.31}
\| e^{itH_{0}}\phi-e^{itH}\phi^{\pm}\|_{L^{2}}
&=\|e^{-itH}e^{itH_{0}}\phi-\phi^{\pm}\|_{L^{2}}\nonumber\\
&=\Big\|\displaystyle\int_{t}^{\pm\infty}e^{-isH}(Vu(s))ds\Big\|_{L^{2}}\nonumber\\
&\lesssim  \|Vu(t)\|_{L_{[t, \pm\infty]}^{2}L_{x}^{\frac{2N}{N+4}}}\rightarrow 0
\end{align}
as $t$ tends to $\pm\infty$.

(ii) Applying $V\geq 0$, \eqref{2.17}, H\"{o}lder inequality, Sobolev embedding and Strichartz estimate \eqref{2.2} in turn
gives
\begin{align}\label{2.32}
&\| e^{-it_{1}H}e^{it_{1}H_{0}}\phi-e^{-it_{2}H}e^{it_{2}H_{0}}\phi\|_{\dot{H}^{2}}
\leq\|H^{\frac{1}{2}}(e^{-it_{1}H}e^{it_{1}H_{0}}\phi-e^{-it_{2}H}e^{it_{2}H_{0}}\phi)\|_{L^{2}}\nonumber\\
&=\Big\|H^{\frac{1}{2}}\displaystyle\int_{t_1}^{t_2}e^{-isH}(Vu(s))ds\Big\|_{L^{2}}
\lesssim  \|\nabla(Vu(t))\|_{L_{[t_{1}, t_{2}]}^{2}L_{x}^{\frac{2N}{N+2}}}\nonumber\\
&\lesssim  \big\|V\big\|_{L_{x}^{\frac{N}{4}}}\ \big\|\nabla u\big\|_{L_{[t_{1}, t_{2}]}^{2}L_{x}^{\frac{2N}{N-6}}}
+\big\||x||\nabla V|\big\|_{L_{x}^{\frac{N}{4}}}\ \big\||x|^{-1} u\big\|_{L_{[t_{1}, t_{2}]}^{2}L_{x}^{\frac{2N}{N-6}}}\nonumber\\
&\lesssim  \Big(\big\|V\big\|_{L_{x}^{\frac{N}{4}}}+\big\||x||\nabla V|\big\|_{L_{x}^{\frac{N}{4}}}\Big)
\ \|\Delta u\|_{L_{[t_{1}, t_{2}]}^{2}L_{x}^{\frac{2N}{N-4}}}\rightarrow 0
\end{align}
as $t_{1}, t_{2}\rightarrow \pm\infty$, where in the last inequality we have used the generalized Hardy equality ($N>8$),
which can be stated as follows (e.g., see Theorem B* in Stein-Weiss \cite{SW}):
Let $1<p<\infty$, $0\leq s<\frac{N}{p}$, then we have
$$
\big\||x|^{-s}f\big\|_{L^{p}({\bf R}^{N})}\lesssim \big\||\nabla|^{s}f\big\|_{L^{p}({\bf R}^{N})}.
$$
Therefore, it follows from \eqref{2.29} and \eqref{2.32} that
the limit of $e^{-itH}e^{itH_{0}}\phi$ exists in $H^{2}$ as
$t$ tends to $\pm\infty$, which is denoted by $\phi^{\pm}$. Repeating the process of (i) gives our desired result \eqref{2.27}.
\end{proof}

\section{Local wellposedness theory and scattering criterion }
\setcounter{equation}{0}

Once established Strichartz type estimates Proposition \ref{lem2.1}, Proposition \ref{lem2.3} and Proposition \ref{lem2.4}, then in this section we will
apply them to obtain local well-posedness
result, small data theory, finite $S(\dot{H}^{s_{c}})$ norm condition on scattering and perturbation lemma for ${\rm BNLS_{V}}$ \eqref{1.1},
whose proofs are similar to the case without potential. Let's first look at local well-posedness.

\begin{lemma}\label{lem2.8} Let $V$, $p$ and $N$ satisfy the assumptions of Theorem \ref{th1.2}. Then
the problem ${\rm BNLS_{V}}$ is locally well-posed in $H^{2}({\bf R}^{N})$.
\end{lemma}

\begin{proof}
  For $M=c\|u_0\|_{H^2}$, we define a map as
\begin{align}\label{2.33}
\Phi(u)=e^{itH}u_0-i\int_{0}^te^{i(t-s)H}|u|^{p-1}u(s)ds,
\end{align}
and a complete metric space as
\begin{align}\label{2.34}
B_{M}=\big\{u\in C(I,H^2): \|\langle\Delta\rangle u\|_{S(L^2, I)}\leq2M
\big\}
\end{align}
with the metric
$$
d(u, v)=\|u-v\|_{S(L^2, I)},
$$
where $I=[0, T]$.

From the Strichartz estimates \eqref{2.7} and \eqref{2.11}, the Sobolev embedding  and the H\"older inequalities,
it follows that for any $u\in B_{M}$,
\begin{align}\label{2.35}
\|\langle\Delta\rangle\Phi(u)\|_{S(L^2, I)}
&\leq c\|u_0\|_{H^{2}}+c\|u\|_{L^{2(p-1)}(I,L^{\frac{N(p-1)}{2}})}^{p-1}\|\langle\nabla\rangle u\|_{L^\infty(I,L^{\frac{2N}{N-2}})}\nonumber\\
&\quad +c\|u\|_{L^{2(p-1)}(I,L^{\frac{N(p-1)}{2}})}^{p-1}\|u\|_{L^\infty(I,L^{2})}\nonumber\\
&\leq c\|u_0\|_{H^{2}}+c\|u\|_{L^{2(p-1)}(I,L^{\frac{N(p-1)}{2}})}^{p-1}\|\langle\Delta\rangle u\|_{L^\infty(I,L^2)}.
\end{align}
If $\frac{4}{p-1}\geq N-4$, using Sobolev embedding and H\"{o}lder inequality, we get that
\begin{align}\label{2.36}
\|u\|_{L^{2(p-1)}(I,L^{\frac{N(p-1)}{2}})}\leq cT^{\frac{1}{2(p-1)}}\ \big\|\langle\nabla\rangle^{\frac{N}{2}-\frac{2}{p-1}}u\big\|_{L^\infty(I,L^2)}
\leq cT^{\frac{1}{2(p-1)}}\ \big\|\langle\Delta\rangle u\big\|_{L^\infty(I,L^2)}.
\end{align}
If $\frac{4}{p-1}< N-4$, the same argument gives
\begin{align}\label{2.37}
\|u\|_{L^{2(p-1)}(I,L^{\frac{N(p-1)}{2}})}
\leq cT^{\frac{1}{(p-1)}-\frac{N-4}{8}}\ \big\|\langle\Delta\rangle u\big\|_{L^{\frac{8(p-1)}{(N-4)(p-1)-4}}(I,L^{\frac{N(p-1)}{2p}})}.
\end{align}
Putting \eqref{2.35}-\eqref{2.37} together yields that
\begin{align}\label{2.38}
\|\langle\Delta\rangle\Phi(u)\|_{S(L^2, I)}
\leq c\|u_0\|_{H^{2}}+cT^{\theta}\|\langle\Delta\rangle u\|_{S(L^2, I)}^{p},
\end{align}
where $\theta=\frac{1}{2}$ if $\frac{4}{p-1}\geq N-4$, and $\theta=1-\frac{(N-4)(p-1)}{8}>0$ if $\frac{4}{p-1}< N-4$.
Similarly, for any $u, v\in B_{M}$,
\begin{align}\label{2.39}
\|\Phi(u)-\Phi(v)\|_{S(L^2, I)}\leq
cT^{\theta}\Big(\|\langle\Delta\rangle u\|_{S(L^2, I)}^{p-1}+\|\langle\Delta\rangle v\|_{S(L^2, I)}^{p-1}\Big)\ \|u-v\|_{S(L^2, I)}.
\end{align}
From a standard argument, we can obtain that if $T$ is
sufficiently small, the map $u\mapsto \Phi(u)$ is a contraction map
on $B_{M}$. Thus, the contraction mapping principle gives a
unique solution $u$ in $B_{M}$.
\end{proof}

In the same way as the local theory, we can obtain the following small data theory.

\begin{proposition}\label{pro2.9} Let $V$ $p$ and $N$ satisfy the assumptions of Theorem \ref{th1.2},
and assume $u_{0}\in H^2({\bf R}^{N})$, $t_0\in I$ an interval of ${\bf R}$. Then
there exists  $\delta_{sd}>0$ such that if
 $\|e^{itH}u_{0}\|_{S(\dot{H}^{s_{c}}, I)}\leq \delta_{sd}, $  there exists a unique solution
 $u\in C(I,H^2({\bf R}^{N}))$ of \eqref{1.1} with initial data $u_0$. Moreover, the solution has conserved mass and  energy,
 and satisfies
\begin{align}\label{2.40}
 \|u\|_{S(\dot{H}^{s_{c}}, I)}\leq
2\delta_{sd},\ \ \ \|u\|_{L^\infty(I,H^2)}\leq c\|u_0\|_{H^2}.
\end{align}
\end{proposition}

\begin{proof}
 For
$\delta=\delta_{sd}$ and $M=c\|u_0\|_{H^2}$, we define a map as
\begin{align}\label{2.41}
\Phi(u)=
e^{i(t-t_0)H_{0}}u_0+i\int_{t_0}^te^{i(t-s)H_{0}}|u|^{p-1}u(s)ds,
\end{align}
and a set as
\begin{align}\label{2.42}
B_{M,\delta}=\big\{v\in C(I,H^2):\ \ \|v\|_{S(\dot{H}^{s_{c}}, I)}\leq2\delta,
\ \ \  \|\langle\Delta\rangle v\|_{L^\infty(I,L^2)}\leq2M \big\}
\end{align}
equipped with the $S(\dot{H}^{s_{c}}, I)$ norm.
Then from the Strichartz estimates \eqref{2.7} and
\eqref{2.11}, using the Sobolev embedding  and H\"older inequality,  we have for any $u\in B_{M,\delta}$,
\begin{align}\label{2.43}
\|\Phi(u)\|_{S(\dot{H}^{s_{c}}, I)}\leq\delta+c\|u\|_{S(\dot{H}^{s_{c}}, I)}^{p},
\end{align}
and
\begin{align}\label{2.44}
\|\langle\Delta\rangle\Phi(u)\|_{L^\infty(I,L^2)}
&\leq c\|\langle\Delta\rangle u_0\|_2+c\|u\|_{L^{2(p-1)}(I,L^{\frac{N(p-1)}{2}})}^{p-1}\|\langle\nabla\rangle u\|_{L^\infty(I,L^{\frac{2N}{N-2}})}\nonumber\\
&+c\|u\|_{L^{2(p-1)}(I,L^{\frac{N(p-1)}{2}})}^{p-1}\| u\|_{L^\infty(I,L^{\frac{2N}{N-2}})}\nonumber\\
 &\leq c\|\langle\Delta\rangle u_0\|_2+c\|u\|_{S(\dot{H}^{s_{c}}, I)}^{p-1}\|\langle\Delta\rangle u\|_{L^\infty(I,L^2)}.
\end{align}
Moreover, for any $u,v\in B_{M,\delta}$,
\begin{align}\label{2.45}
\|\Phi(u)-\Phi(v)\|_{S(\dot{H}^{s_{c}}, I)}\leq c\Big(\|u\|_{S(\dot{H}^{s_{c}}, I)}^{p-1}
+\|v\|_{S(\dot{H}^{s_{c}}, I)}^{p-1}\Big)\ \|u-v\|_{S(\dot{H}^{s_{c}}, I)}.
\end{align}
From a standard argument, we can obtain that if $\delta$ is
sufficiently small, the map $u\mapsto \Phi(u)$ is a contraction map
on $B_{M,\delta}$. Thus,  the contraction mapping principle gives a
unique solution $u$ in $B_{M,\delta}$ satisfying \eqref{2.40}.
\end{proof}

Now we turn to use a similar argument as in \cite{Pau3} to establish the following scattering result,
which can be combined with Proposition \ref{pro2.9}
to get a scattering result of small data.

\begin{proposition}\label{pro2.10}
Let $V$, $p$ and $N$ satisfy the assumptions of Theorem \ref{th1.2}. If $u(t)\in C({\bf R},H^2({\bf R}^{N}))$ be a solution of \eqref{1.1} such that $\sup_{t\in {\bf R}}\|u(t)\|_{H^{2}}<\infty$. If
 $\|u\|_{S(\dot{H}^{s_{c}})}<\infty$,
then $u(t)$ scatters in $H^2({\bf R}^{N})$.
That is , there exists $\phi^{\pm}\in H^2({\bf R}^{N})$ such that
$$
\lim_{t\rightarrow\pm\infty}\|u(t)-e^{itH}\phi^{\pm}\|_{H^2({\bf R}^{N})}=0.
$$
\end{proposition}

\begin{proof}
 We claim that
\begin{align}\label{2.46}
\phi^{\pm}:=u_{0}-i\displaystyle\int_{0}^{\pm\infty}e^{-isH}(|u|^{p-1}u)(s)ds
\end{align}
exist in $H^{2}$. Indeed, using $V\geq 0$, Propostion \ref{lem2.1} and Proposition \ref{lem2.4} gives
\begin{align}\label{2.47}
\Big\|\displaystyle\int_{t_{1}}^{t_{2}}e^{-isH}(|u|^{p-1}u)(s)ds\Big\|_{H^{2}}
&\lesssim \Big\|H_{0}^{\frac{1}{2}}\displaystyle\int_{t_{1}}^{t_{2}}e^{-isH}(|u|^{p-1}u)(s)ds\Big\|_{L^{2}}\nonumber\\
&\quad+\Big\|\displaystyle\int_{t_{1}}^{t_{2}}e^{-isH}(|u|^{p-1}u)(s)ds\Big\|_{L^{2}}\nonumber\\
&\lesssim \big\|\nabla(|u|^{p-1}u)\big\|_{L_{[t_{1}, t_{2}]}^{2}L_{x}^{\frac{2N}{N+2}}}
+\big\||u|^{p-1}u\big\|_{L_{[t_{1}, t_{2}]}^{2}L_{x}^{\frac{2N}{N+4}}}\nonumber\\
&\lesssim \|u\|_{L_{[t_{1}, t_{2}]}^{2(p-1)}L_{x}^{\frac{N(p-1)}{2}}}^{p-1}
\|\nabla u\|_{L_{[t_{1}, t_{2}]}^{\infty}L_{x}^{\frac{2N}{N-2}}}\nonumber\\
&\quad+\|u\|_{L_{[t_{1}, t_{2}]}^{2(p-1)}L_{x}^{\frac{N(p-1)}{2}}}^{p-1}
\|\nabla u\|_{L_{[t_{1}, t_{2}]}^{\infty}L_{x}^{2}}\nonumber\\
&\lesssim  \|u\|_{S(\dot{H}^{s_{c}}, [t_{1}, t_{2}])}^{p-1}\|u(t)\|_{H^{2}}\rightarrow 0,
\end{align}
as $t_{1}, t_{2}$ tend to $\pm\infty$.

Hence, $\phi^{\pm}$ is well defined. Then, using \eqref{2.46} and repeating the above estimates again, we obtain that
\begin{align}\label{2.47}
\|u(t)-e^{itH}\phi^{\pm}\|_{H^2}
&=\Big\|\displaystyle\int_{0}^{\pm\infty}e^{-isH}(|u|^{p-1}u)(s)ds\Big\|_{H^2}\nonumber\\
&\lesssim  \|u\|_{S(\dot{H}^{s_{c}}, [t, \infty])}^{p-1}\sup_{t\in{\bf R}}\|u(t)\|_{H^{2}}
\rightarrow 0,
\end{align}
as $t$ tends to $\pm\infty$.
\end{proof}

Finally, we state a useful perturbation lemma, whose proof shall be omitted, since it is similar to that for \cite{Guo}.

\begin{lemma}\label{lem2.11}
Let $V$, $p$ and $N$ satisfy the assumptions of Theorem \ref{th1.2}.
 Then for any given $A$, there exist $\epsilon_0=\epsilon_0(A,n,p)$ and
$c=c(A)$ such that for any $\epsilon\leq\epsilon_0$, any interval $I
=(T_1,T_2)\subset {\bf R}$ and any
 $\tilde{u}=\tilde{u}(x,t)\in H^2$ satisfying
$$
 i\tilde{u}_{t}+H \tilde{u}-|\tilde{u}|^{p-1}\tilde{u}=e,
 $$
if for some $(q,r)\in\Lambda_{-s_c}$,
 $$
 \|\tilde{u}\|_{S(\dot{H}^{s_{c}}, I)}\leq A,\ \   \|e\|_{L^{q'}(I;L^{r'})}\leq \epsilon
 $$
and
$$
\|e^{i(t-t_0)H}(u(t_0)-\tilde{u}(t_0)\|_{S(\dot{H}^{s_{c}}, I)}\leq
\epsilon,
$$
then the solution $u\in C(I;H^2)$ of \eqref{1.1}
satisfying
 $$
 \|u-\tilde u\|_{S(\dot{H}^{s_{c}}, I)}\leq c(A)\epsilon.
 $$
 \end{lemma}

\section{Sharp Gagliardo-Nirenberg inequality}\label{sec-5}
\setcounter{equation}{0}

In this section,  under the assumptions of Theorem \ref{th1.2}, we will find a minimizing sequence of the nonlinear functional
\begin{align}\label{3.1}
J_{V}(u)=\frac{\|u\|_{L^2}^{p+1-\frac{N(p-1)}{4}}(\|\Delta u\|_{L^2}^{2}
+\displaystyle\int_{{\bf R}^{N}}V|u|^{2}dx)^{\frac{N(p-1)}{8}}}{\| u\|_{L^{p+1}}^{p+1}}.
\end{align}
It's known from \cite{Guo, Zhu} that for $V=0$, $J_{0}(u)$ attains its minimum $J_{0}$  at $u=Q(x)\geq 0$, which solves the equation \eqref{1.4}, and
\begin{align}\label{3.2}
J_{0}=J_{0}(Q)=\frac{\|Q\|_{L^2}^{p+1-\frac{N(p-1)}{4}}\|\Delta Q\|_{L^2}^{\frac{N(p-1)}{4}}}{\|Q\|_{L^{p+1}}^{p+1}},
\end{align}
which together with the identities
\begin{align}\label{3.3}
\|\Delta Q\|_{L^2}^2=\frac{N(p-1)}{4(p+1)}\|Q\|_{L^{p+1}}^{p+1},\ \
\|Q\|_2^2=\frac{p-1}{2(p+1)}\|Q\|_{L^{p+1}}^{p+1},\ \
E_0(Q)=\frac{N(p-1)-8}{8(p+1)}\|Q\|_{L^{p+1}}^{p+1},
\end{align}
implies that the best constant of the Gagliardo-Nirenberg inequality
\begin{align}\label{3.4}
\|u\|_{L^{p+1}}^{p+1}\leq C_{GN} \|u\|_{L^2}^{p+1-\frac{N(p-1)}{4}}\|\Delta
u\|_{L^2}^{\frac{N(p-1)}{4}}
\end{align}
is
\begin{align}\label{3.5}
 C_{GN}=\frac
1J_{0}=\frac{4(p+1)}{N(p-1)}\frac1{\|Q\|_{L^2}^{p+1-\frac{N(p-1)}{4}}\|\Delta
Q\|_{L^2}^{\frac{N(p-1)}{4}-2} }
\end{align}

\begin{lemma}\label{lem3.1}
If $V\geq 0$, then $\{Q(\cdot-n)\}_{n=1}^{\infty}$ is a minimizing
sequence for $J_{V}(u)$.
\end{lemma}

\begin{proof}
it follows from \eqref{3.2}, \eqref{3.4} and \eqref{3.5} that
\begin{align}\label{3.6}
J_{0}(Q)\leq J_{0}(u).
\end{align}
On the one hand,
\begin{align}\label{3.7}
\lim_{n\rightarrow\infty}J_{V}(Q(\cdot-n))=J_{0}(Q),
\end{align}
where we used the inequality
\begin{align}\label{3.8}
\displaystyle\int_{{\bf R}^{N}}V(x)Q(x-n)^{2}dx\lesssim \|V\|_{L^{\frac{N}{4}}}\|Q\|_{L^{\frac{2N}{N-4}}}^{2}
\lesssim \|V\|_{L^{\frac{N}{4}}}\|\Delta Q\|_{L^2}^{2}
\end{align}
On the other hand, for $V\geq 0$, it is easy to see that
\begin{align}\label{3.9}
J_{0}(u)\leq J_{V}(u)
\end{align}
Putting \eqref{3.6}, \eqref{3.7} and \eqref{3.9} together yields that
\begin{align}\label{3.10}
\lim_{n\rightarrow\infty}J_{V}(Q(\cdot-n))\leq J_{V}(u).
\end{align}
Thus, we get our desired result.
\end{proof}

\begin{remark}\label{rem3.2}
It follows from lemma 3.1 that $J_0(Q)=\lim_{n\rightarrow\infty}J_{V}(Q(\cdot-n))\leq J_V(u)$ holds
for any $u$, which implies that the following sharp inequality holds: 
\begin{align}\label{3.11}
\|u\|_{L^{p+1}}^{p+1}\leq C_{GN} \|u\|_{L^2}^{p+1-\frac{N(p-1)}{4}}\|H^{\frac12}
u\|_{L^2}^{\frac{N(p-1)}{4}},
\end{align}
where $ C_{GN}$ is the same Gagliardo-Nirenberg constant \eqref{3.5}. \end{remark}

\section{Criteria for global well-posedness}
\setcounter{equation}{0}
In this section we first give a criteria for global well-posedness, but we omit its proof, since
it is similar to that of Theorem 4.1 in \cite{Guo}. Indeed, it suffices to use the previous section's result \eqref{3.11} and
replace $\Delta$ by $H^{\frac{1}{2}}$ in the proof.

\begin{theorem}\label{th4.1}  Let $V$, $p$ and $N$ the assumptions of Theorem \ref{th1.2} hold, $u_{0}\in H^{2}({\bf R}^{N})$ and  $I=(T_{-},T_{+})$  be the
maximal time interval of existence of $u(t)$ solving \eqref{1.1}.
Suppose that
\begin{align}\label{4.1}
M(u)^{\frac{2-s_c}{s_c}}E(u)<M(Q)^{\frac{2-s_c}{s_c}}E_{0}(Q).
\end{align}
If \eqref{4.1} holds and
\begin{align}\label{4.2}
\|u_{0}\|_{L^{2}}^{\frac{2-s_c}{s_c}}\|H^{\frac{1}{2}} u_{0}\|_{L^{2}}<\|Q\|_{L^{2}}^{\frac{2-s_c}{s_c}}\|\Delta Q\|_{L^{2}},
\end{align}
then $I=(-\infty,+\infty)$, i.e., the solution exists globally in
time, and for all time ~$t\in {\bf{R}},$~
\begin{align}\label{4.3}
\|u(t)\|_{L^{2}}^{\frac{2-s_c}{s_c}}\|H^{\frac{1}{2}} u(t)\|_{L^{2}}<\|Q\|_{L^{2}}^{\frac{2-s_c}{s_c}}\|\Delta Q\|_{L^{2}}.
\end{align}
If \eqref{4.1} holds and
\begin{align}\label{4.4}
\|u_{0}\|_{L^{2}}^{\frac{2-s_c}{s_c}}\|H^{\frac{1}{2}} u_{0}\|_{L^{2}}>\|Q\|_{L^{2}}^{\frac{2-s_c}{s_c}}\|\Delta Q\|_{L^{2}},
\end{align}
then for $t\in I,$
\begin{align}\label{4.5}
\|u(t)\|_{L^{2}}^{\frac{2-s_c}{s_c}}\|H^{\frac{1}{2}} u(t)\|_{L^{2}}>\|Q\|_{L^{2}}^{\frac{2-s_c}{s_c}}\|\Delta Q\|_{L^{2}}.
\end{align}
\end{theorem}

\begin{remark}\label{rem4.2} From the proof of Theorem \ref{th4.1}, we conclude that if the condition \eqref{4.2} holds, then
there exists $\delta>0$ such that
$M(u)^{\frac{2-s_c}{s_c}}E(u)<(1-\delta)M(Q)^{\frac{2-s_c}{s_c}}E_{0}(Q)$,
and thus there exists $\delta_0=\delta_0(\delta)$ such that
$\|u(t)\|_{L^{2}}^{\frac{2-s_c}{s_c}}\|H^{\frac{1}{2}} u(t)\|_{L^{2}}<
(1-\delta_0)\|Q\|_{L^{2}}^{\frac{2-s_c}{s_c}}\|\Delta Q\|_{L^{2}}.$
\end{remark}

The next two lemmas provide some additional properties for the solution $u$  under the
hypotheses \eqref{4.1} and \eqref{4.2} of Theorem \ref{th4.1}.
These lemmas
 will be needed in the proof of Theorem \ref{th1.2} through a virial-type estimate, which will be established  in the last
two sections.

\begin{lemma}\label{lem4.3}
In the situation of Theorem \ref{th4.1},
take $\delta>0$ such that
$M(u_0)^{\frac{2-s_c}{s_c}}E(u_0)<(1-\delta)M(Q)^{\frac{2-s_c}{s_c}}E_{0}(Q)$.
If $u$ is a solution of the problem \eqref{1.1} with initial data $u_0$,
then there exists $C_\delta>0$ such that for all $t\in{\bf R}$,
\begin{align}\label{4.6}
\|\Delta u\|_{L^{2}}^2-\frac{N(p-1)}{4(p+1)}\|u\|_{L^{p+1}}^{p+1}\geq
C_\delta\|\Delta u\|_2^2.
\end{align}
\end{lemma}

\begin{proof}
 By Remark \ref{rem4.2}, there exists $\delta_0=\delta_0(\delta)$ such that
\begin{align}\label{4.7}
\|u(t)\|_{L^{2}}^{\frac{2-s_c}{s_c}}\|H^{\frac{1}{2}} u(t)\|_{L^{2}}<
(1-\delta_0)\|Q\|_{L^{2}}^{\frac{2-s_c}{s_c}}\|\Delta Q\|_{L^{2}}.
\end{align}
Since $V$ is nonnegative, it is obvious that
\begin{align}\label{4.8}
\|\Delta u\|_{L^{2}}\leq \|H^{\frac{1}{2}} u\|_{L^{2}},
\end{align}
which combined with \eqref{4.7} yields that
\begin{align}\label{4.9}
\|u(t)\|_{L^{2}}^{\frac{2-s_c}{s_c}}\|\Delta u(t)\|_{L^{2}}<
(1-\delta_0)\|Q\|_{L^{2}}^{\frac{2-s_c}{s_c}}\|\Delta Q\|_{L^{2}}.
\end{align}
The remaining proof is the same as that for Lemma 4.2 in \cite{Guo}.
\end{proof}

The following lemma is about the comparability of the gradient and the total energy, and we
omit its proof as well, since we only replace $\Delta$ by $H^{\frac{1}{2}}$ in the proof of Lemma 4.3 in \cite{Guo}.

\begin{lemma}\label{lem4.4}
In the situation of Theorem \ref{th4.1}, we have
\begin{align}\label{4.10}
\frac{N(p-1)-8}{2N(p-1)}\|H^{\frac{1}{2}} u(t)\|_{L^{2}}^2\leq E(u)\leq\frac{1}{2}\|H^{\frac{1}{2}} u(t)\|_{L^{2}}^2.
\end{align}
\end{lemma}

Finally, we give the result about existence of wave operators, which will be used to established the scattering theory.

\begin{proposition}\label{pro4.5} Under the assumptions of Theorem \ref{th1.2}, and
 suppose $ \psi^\pm\in H^2({\bf R}^{N})$ and
\begin{align}\label{4.11}
\frac{1}{2}\|\psi^\pm\|_{L^{2}}^{\frac{2(2-s_c)}{s_c}}\|H^{\frac{1}{2}}\psi^\pm\|_{L^{2}}^2<E_{0}(Q)M(Q)^{\frac{2-s_c}{s_c}}.
\end{align}
Then there exists $v_0\in H^2({\bf R}^{N})$ such that the solution $v$ of
\eqref{1.1} with initial data $v_0$ obeys the assumptions \eqref{4.1} and \eqref{4.2} and satisfies
\begin{align}\label{4.12}
\lim_{t\rightarrow\pm\infty}\|v(t)-e^{itH}\psi^\pm\|_{H^2({\bf R}^{N})}=0.
\end{align}
\end{proposition}

\begin{proof}
 Similar to the proof of the small data scattering theory Proposition \ref{pro2.9},
we can solve the integral equation
\begin{align}\label{4.13}
v(t)=e^{itH}\psi^+ +i\int_t^\infty e^{i(t-s)H}|v|^{p-1}v(s)ds
\end{align}
for $t\geq T$ with $T$ large.

 In fact, there exists some large $T$ such that
$\|e^{itH}\psi^+\|_{S(\dot{H}^{s_{c}}, [T,\infty))}\leq \delta_{sd},$
where $\delta_{sd}$ is defined by Proposition \ref{pro2.9}.
Then, the same arguments as used in Proposition \ref{pro2.9} give a solution
$v\in C([T,\infty),H^2)$ of \eqref{4.13}.
Moreover, we also have
\begin{align}\label{4.14}
\|v\|_{S(\dot{H}^{s_{c}}, [T,\infty))}\leq2\delta_{sd},\quad
{\rm and}\quad
 \|v\|_{L^{\infty}([T,\infty); H^2)} \leq c\|v_0\|_{H^{2}}.
\end{align}
Thus by Proposition \ref{lem2.1}, Proposition \ref{lem2.4},
\begin{align}\label{4.15}
&\big\| v-e^{itH}\psi^+\big\|_{L_{[T,\infty)}^{\infty}H_{x}^{2}}
=\left\|\int_t^\infty e^{i(t-s)H}|v|^{p-1}v(s)ds\right\|_{L_{[T,\infty)}^{\infty}H_{x}^{2}}\nonumber\\
&\leq \left\|H_{0}^{\frac{1}{2}}\int_t^\infty e^{i(t-s)H}(|v|^{p-1}v)(s)ds\right\|_{L_{[T,\infty)}^{\infty}L_{x}^{2}}
+\left\|\int_t^\infty e^{i(t-s)H}(|v|^{p-1}v)(s)ds\right\|_{L_{[T,\infty)}^{\infty}L_{x}^{2}}\nonumber\\
&\leq c\|v\|_{L_{[T, \infty)}^{2(p-1)}L_{x}^{\frac{N(p-1)}{2}}}^{p-1}
\|\nabla v\|_{L_{[T, \infty)}^\infty L_{x}^{\frac{2N}{N-2}}}
+c\|v\|_{L_{[T, \infty)}^{2(p-1)}L_{x}^{\frac{N(p-1)}{2}}}^{p-1}\|v\|_{L_{[T, \infty)}^\infty L_{x}^{2}}\nonumber\\
&\leq c\|v\|_{S(\dot{H}^{s_{c}}, [T, \infty))}^{p-1}\| v\|_{H_{x}^2}.
\end{align}
we get that
\begin{align}\label{4.16}
\| v-e^{itH}\psi^+\|_{L_{[T,\infty)}^{\infty}H_{x}^{2}} \rightarrow 0
\ \  {\rm as} ~T\rightarrow \infty,
\end{align}
which implies $v(t)-e^{itH}\psi^+\rightarrow0 $ in $H^2({\bf R}^{N})$ as $t\rightarrow+\infty$. Thus by Sobolev embedding, we obtain that
$v(t)-e^{itH}\psi^+\rightarrow0 $ in $L_{x}^{p+1}({\bf R}^{N})$ as $t\rightarrow+\infty$, which implies that
$\lim_{t\rightarrow+\infty}E(v(t))=\lim_{t\rightarrow+\infty}E(e^{itH}\psi^+)$. Thus, in view of \eqref{4.11}, we obtain that
\begin{align}\label{4.17}
M(v(T))^{\frac{ 2-s_c }{s_c}}E(v(T))
&= \lim_{t\rightarrow+\infty}M(v(t))^{\frac{ 2-s_c }{s_c}}E(v(t))\nonumber\\
&= \lim_{t\rightarrow+\infty}M(e^{itH}\psi^+)^{\frac{ 2-s_c }{s_c}}E(e^{itH}\psi^+)\nonumber\\
&=\lim_{t\rightarrow+\infty}\|\psi^+\|_{L^{2}}^{\frac{ 4-2s_c }{s_c}}
\Big(\frac{1}{2}\|H^{\frac{1}{2}}\psi^+\|_{L^2}^{2}-\frac{1}{p+1}\|e^{itH}\psi^+\|_{L_{x}^{p+1}}^{p+1}\Big)\nonumber\\
&\leq\frac{1}{2}\|\psi^+\|_{L^{2}}^{\frac{2(2-s_c)}{s_c}}\|H^{\frac{1}{2}}\psi^\pm\|_{L^2}^{2}<E_{0}(Q)M(Q)^{\frac{2-s_c}{s_c}}.
\end{align}
Moreover, we note that
\begin{align}\label{4.18}
&\lim_{t\rightarrow\infty}\|v(t)\|_{L_{x}^2}^{\frac{2(2-s_c)}{s_c}}\|H^{\frac{1}{2}} v(t)\|_{L_{x}^2}^2
=\lim_{t\rightarrow\infty}\|e^{itH}\psi^+\|_{L_{x}^2}^{\frac{2(2-s_c)}{s_c}}\|H^{\frac{1}{2}} e^{itH}\psi^+\|_{L_{x}^2}^{2}\nonumber\\
&=\|\psi^+\|_{L_{x}^2}^{\frac{2(2-s_c)}{s_c}}\|H^{\frac{1}{2}}\psi^+\|_{L_{x}^2}^2<2E_{0}(Q)M(Q)^{\frac{2-s_c}{s_c}}\nonumber\\
&=\frac{N(p-1)-8}{N(p-1)}\|Q\|_{L^{2}}^{\frac{2(2-s_c)}{s_c}}\|\Delta Q\|_{L^{2}}^2
<\|Q\|_{L^{2}}^{\frac{2(2-s_c)}{s_c}}\|\Delta Q\|_{L^{2}}^2.
\end{align}
Hence, for sufficiently large $T$, $v(T)$ satisfies \eqref{4.1} and \eqref{4.2}, which implies that $v(t)$ is a global solution in $H_{x}^{2}({\bf R}^{N})$.
Thus, we can evolve $v(t)$ from $T$ back to the initial time 0. By the same way, we can show \eqref{4.12} for negative time.
\end{proof}

\section{Existence and compactness of a critical element}
\setcounter{equation}{0}

\begin{definition}\label{def5.1}
We say that $SC(u_0)$ holds if for $u_0\in H^2({\bf R}^{N})$ satisfying
$$\|u_{0}\|^{\frac{2-s_c}{s_c}}_{L^{2}}\|H^{\frac{1}{2}}
u_{0}\|_{L^{2}}<\|Q\|^{\frac{2-s_c}{s_c}}_{L^{2}}\|\Delta Q\|_{L^{2}}$$ and
$$E(u_{0})M(u_{0})^{\frac{2-s_c}{s_c}}<E_{0}(Q)M(Q)^{\frac{2-s_c}{s_c}}, $$
 the corresponding solution $u$  of \eqref{1.1} with the maximal
interval of existence $I=(-\infty,+\infty)$ satisfies
\begin{align}\label{5.1}
\|u\|_{S(\dot{H}^{s_{c}})}<+\infty.
\end{align}
\end{definition}

We first claim  that  there exists $\delta>0$ such that if
$$E(u)M(u)^{\frac{2-s_c}{s_c}}<\delta, \ \ \
\|u_{0}\|^{\frac{2-s_c}{s_c}}_{L^{2}}\|H^{\frac{1}{2}}
u_{0}\|_{L^{2}}<\|Q\|^{\frac{2-s_c}{s_c}}_{L^{2}}\|\Delta Q\|_{L^{2}},$$ then
\eqref{5.1} holds.
In fact, by Proposition \ref{lem2.1},
the norm equivalence Remark \ref{rem1.3} and \eqref{4.10}, we have
\begin{align}\label{5.2}
\|e^{itH}u_{0}\|_{S(\dot{H}^{s_{c}})}^{\frac{2}{s_{c}}}
\lesssim  \||\nabla|^{s_{c}}u_{0}\|_{L^{2}}^{\frac{2}{s_{c}}}
\lesssim\|u_{0}\|_{L^{2}}^{\frac{4-2s_{c}}{s_{c}}}\|\Delta u_{0}\|_{L^{2}}^{2}
\sim \|u_{0}\|_{L^{2}}^{\frac{4-2s_{c}}{s_{c}}}\|H^{\frac{1}{2}} u_{0}\|_{L^{2}}^{2}
\sim E(u_{0})M(u_{0})^{\frac{2-s_c}{s_c}}.
\end{align}
Hence, it follows from Proposition \ref{pro2.9} and Proposition \ref{pro2.10} that \eqref{5.1} holds for all
sufficiently small $\delta>0$.

Now for each $\delta>0$, we define the set $S_\delta$ to be the collection
of all such initial data in $H^2$ :
\begin{align}\label{5.3}
S_\delta=\Big\{u_0\in H^2({\bf R}^{N}):\ \  E(u)M(u)^{\frac{2-s_c}{s_c}}<\delta \ \  and \ \
\|u_{0}\|^{\frac{2-s_c}{s_c}}_{2}\|H^{\frac{1}{2}} u_{0}\|_{2}<\|Q\|^{\frac{2-s_c}{s_c}}_{2}\|\Delta Q\|_{2}\Big\}.
\end{align}
We also define
\begin{align}\label{5.4}
(M^{\frac{2-s_c}{s_c}}E)_c=\sup\big\{\delta:\ \ u_0\in S_\delta\Rightarrow SC(u_0)\ \  holds \big \}.
\end{align}
If $(M^{\frac{2-s_c}{s_c}}E)_c=M(Q)^{\frac{2-s_c}{s_c}}E_{0}(Q)$, then we are done. Thus we assume now
\begin{align}\label{5.5}
(M^{\frac{2-s_c}{s_c}}E)_c<M(Q)^{\frac{2-s_c}{s_c}}E_{0}(Q).
\end{align}
Our goal in this section is to show the existence of an $H^2({\bf R}^{N})$
solution $u_c$ of \eqref{1.1} with the initial data $u_{c,0}$ such that
\begin{align}\label{5.6}
\| u_{c,0}\|^{\frac{2-s_c}{s_c}}_{L^{2}}\|H^{\frac{1}{2}}u_{c,0}\|_{L^{2}}
<\|Q\|^{\frac{2-s_c}{s_c}}_{L^{2}}\|\Delta Q\|_{L^{2}},
\end{align}
\begin{align}\label{5.7}
M(u_c)^{\frac{2-s_c}{s_c}}E(u_c)= (M^{\frac{2-s_c}{s_c}}E)_c
\end{align}
and
$ SC(u_{c,0})$ does not hold. Moreover, we  show that if
$\|u_c\|_{S(\dot{H}^{s_{c}})}=\infty$, then
$K=\{u_c(x,t)|t\in {\bf R}\}$ is precompact in  $H^2({\bf R}^{N})$.

Prior to fulfilling  our main task, we  first establish the decay property for the semigroup
$e^{itH}\phi$ in $L^{p+1}$, where $1<p<1+\frac{8}{N-4}$ and $\phi\in H^{2}({\bf R}^{N})$. It was well-known that
$L^1$-$L^\infty$ estimates of $e^{itH}\phi$ can imply the decay property. However, as far as we know,
 there are not our required dispersive estimate at present. Hence, we shall give a detailed proof
of the decay property.

\begin{lemma}\label{lem5.2}
 $\lim_{t\rightarrow \infty}\|e^{itH}\phi\|_{L^{p+1}}=0$ for $1<p<1+\frac{8}{N-4}$ and
$\phi\in H^{2}({\bf R}^{N})$.
\end{lemma}

\begin{proof}
In view of the  the Strichartz estimates,
\begin{align}\label{5.8}
\|e^{itH}\phi\|_{L_{t}^{\frac{8(p+1)}{N(p-1)}}L_{x}^{p+1}}\leq \|\phi\|_{L^{2}},
\end{align}
it suffices to prove that $\lim_{t\rightarrow \infty}\|e^{itH}\phi\|_{L^{p+1}}$ exists.
To this end,  we need to  show that the map $e^{itH}\phi: t\mapsto L^{p+1}$ is uniformly bounded
and uniformly continuous. Uniformly boundedness can be followed from Sobolev embedding and the equivalence norm
Remark \ref{rem1.3},
that is, for $1<p<1+\frac{8}{N-4}$,
\begin{align}\label{5.9}
\|e^{itH}\phi\|_{L^{p+1}}\lesssim \|e^{itH}\phi\|_{H^{2}}\lesssim \|\phi\|_{H^{2}}.
\end{align}
On the other hand, for any $t_{1}, t_{2}\in {\bf R}$, applying Gagliardo-Nrenberg inequality and the equivalence norm gives
\begin{align}\label{5.10}
 \big\|e^{it_{1}H}\phi-e^{it_{2}H}\phi\big\|_{L^{p+1}}
 &\lesssim \big \|e^{it_{1}H}\phi-e^{it_{2}H}\phi\big\|_{L^{2}}^{1-\frac{N(p-1)}{4(p+1)}}\
\big\|\Delta(e^{it_{1}H}\phi-e^{it_{2}H}\phi)\big\|_{L^{2}}^{\frac{N(p-1)}{4(p+1)}}\nonumber\\
&\lesssim  \big\|e^{it_{1}H}\phi-e^{it_{2}H}\phi\big\|_{L^{2}}^{1-\frac{N(p-1)}{4(p+1)}}
\ \big\|\phi\big\|_{H^{2}}^{\frac{N(p-1)}{4(p+1)}}
\end{align}
As $e^{it_{1}H}$ is a strongly continuous semigroup in $L^{2}$, $e^{itH}\phi: t\mapsto L^{2}$
is an uniformly continuous functional. So it follows from \eqref{5.10} that $e^{itH}\phi: t\mapsto L^{p+1}$
is also an uniformly continuous functional.
\end{proof}

Next, following the idea of Hong \cite{Hong}, we establish
 the linear decomposition  associated with a perturbed linear propagator $e^{itH_{r_{n}}}$,
 where $$H_{r_{n}}=\Delta^2+\frac1{r_n^4}V\Big(\frac1{r_n}\Big).$$
 The profile decomposition
associated with the free linear propagator $e^{it\Delta}$ was established in \cite{Duy, Holmer} by using the concentration
compactness principle in the spirit of Keraani \cite{Ker} and Kenig and Merle \cite{Kenig}.
We refer to \cite{J-P-S} for the linear profile decomposition
for the one-dimensional fourth-order Schr\"{o}dinger equaiton.

\begin{proposition}\label{pro5.3} Let $V$, $p$ and $N$  satisfy the assumptions of Theorem \ref{th1.2}, $\phi_{n}(x)$ be radial and uniformly
bounded  in $H^{2}({\bf R}^{N})$, and  $r_{n}=1, r_{n}\rightarrow 0$ or $r_{n}\rightarrow\infty$.
Then for each $M$ there exists a
subsequence of $\phi_{n}$, which is denoted by itself, such that the
following statements hold.

(i) For each $1\leq j\leq M$, there exists
(fixed in n) a radial profile $\psi^{j}(x)$ in $H^2({\bf R}^{N})$ and
 a sequence (in $n$) of time
shifts $t_{n}^{j}$, and there exists a sequence (in $n$) of
remainders $W_{n}^{M}(x)$ in $H^2({\bf R}^{N})$  such that
\begin{align}\label{5.11}
\phi_{n}(x)=\sum_{j=1}^{M}e^{-it_{n}^{j}H_{r_{n}}}\psi^{j}(x)+W_{n}^{M}(x).
\end{align}

(ii) The time  sequences have a pairwise divergence property, i.e.,  for $1\leq j\neq k\leq M$,
\begin{align}\label{5.12}
\lim_{n\rightarrow+\infty}
|t_{n}^{j}-t_{n}^{k}|=+\infty.
\end{align}

(iii) The remainder sequence has the following asymptotic smallness
property:
\begin{align}\label{5.13}
\lim_{M\rightarrow+\infty}\Big(\lim_{n\rightarrow+\infty}\|e^{itH_{r_{n}}}W_{n}^{M}\|_{S(\dot{H}^{s_{c}})}\Big)=0.
\end{align}

(iv) For each fixed $M$, we have the asymptotic
Pythagorean expansion as follows
\begin{align}\label{5.14}
\|\phi_{n}\|_{L^{2}}^{2}
=\sum_{j=1}^{M}\|\psi^{j}\|_{L^{2}}^{2}+\|W_{n}^{M}\|_{L^{2}}^{2}+o_{n}(1),
\end{align}
\begin{align}\label{5.15}
\|H_{r_{n}}^{\frac{1}{2}}\phi_{n}\|_{L^{2}}^{2}
=\sum_{j=1}^{M}\|H_{r_{n}}^{\frac{1}{2}}\psi^{j}\|_{L^{2}}^{2}
+\|H_{r_{n}}^{\frac{1}{2}}W_{n}^{M}\|_{L^{2}}^{2}+o_{n}(1),
\end{align}
where $o_{n}(1)\rightarrow0$ as $n\rightarrow+\infty$.
\end{proposition}

\begin{proof}
 Let's first consider the case $r_{n}\rightarrow 0$ or $r_{n}\rightarrow\infty$.
According to Lemma 5.3 of the fist author \cite{Guo}, there exists a subsequence of $\phi_{n}$, which
is still denoted by itself, such that
\begin{align}\label{5.16}
\phi_{n}(x)=\sum_{j=1}^{M}e^{-it_{n}^{j}H_{0}}\psi^{j}(x)+W_{n}^{M}(x).
\end{align}
In order to get the form of \eqref{5.11}, we can rewrite \eqref{5.16} as
\begin{align}\label{5.17}
\phi_{n}(x)=\sum_{j=1}^{M}e^{-it_{n}^{j}H_{r_{n}}}\psi^{j}(x)+{\widetilde W}_{n}^{M}(x),
\end{align}
where
\begin{align}\label{5.18}
{\widetilde W}_{n}^{M}(x)=W_{n}^{M}(x)
+\sum_{j=1}^{M}\Big(e^{-it_{n}^{j}H_{0}}\psi^{j}(x)-e^{-it_{n}^{j}H_{r_{n}}}\psi^{j}(x)\Big).
\end{align}
Now we start verifying that \eqref{5.17} satisfies the properties \eqref{5.12}-\eqref{5.15}. It's obvious that \eqref{5.12} is true,
so let's look at \eqref{5.13}.
Applying the formula \eqref{2.28} to $e^{itH_{r_{n}}}W_{n}^{M}$ yields that
\begin{align}\label{5.19}
\|e^{itH_{r_{n}}}W_{n}^{M}\|_{S(\dot{H}^{s_{c}})}
&\leq \|e^{itH_{0}}W_{n}^{M}\|_{S(\dot{H}^{s_{c}})}
+\Big\|\displaystyle\int_{0}^{t}e^{i(t-s)H_{r_{n}}}(V_{r_{n}}e^{isH_{0}}W_{n}^{M})ds\Big\|_{S(\dot{H}^{s_{c}})}\nonumber\\
&\lesssim \|e^{itH_{0}}W_{n}^{M}\|_{S(\dot{H}^{s_{c}})}
+\|V_{r_{n}}e^{itH_{0}}W_{n}^{M}\|_{L_{t}^{\frac{4}{2-s_{c}}}L_{x}^{\frac{2N}{N+4}}}\nonumber\\
&\lesssim \|e^{itH_{0}}W_{n}^{M}\|_{S(\dot{H}^{s_{c}})}
+\|V_{r_{n}}\|_{L^{\frac{N}{4}}}\|e^{itH_{0}}W_{n}^{M}\|_{L_{t}^{\frac{4}{2-s_{c}}}L_{x}^{\frac{2N}{N-4}}}\nonumber\\
&=(1+\|V\|_{L^{\frac{N}{4}}})\|e^{itH_{0}}W_{n}^{M}\|_{S(\dot{H}^{s_{c}})}\rightarrow 0,
\end{align}
as $n\rightarrow\infty$ and $M\rightarrow\infty$.

Using the same argument to $e^{-it_{n}^{j}H_{0}}\psi^{j}(x)-e^{-it_{n}^{j}H_{r_{n}}}\psi^{j}(x)$, we obtain
\begin{align}\label{5.20}
&\|e^{itH_{r_{n}}}(e^{-it_{n}^{j}H_{0}}\psi^{j}-e^{-it_{n}^{j}H_{r_{n}}}\psi^{j})\|_{S(\dot{H}^{s_{c}})}\nonumber\\
&=\Big\|\displaystyle\int_{-t_{n}^{j}}^{0}
e^{i(t-t_{n}^{j}-s)H_{r_{n}}}(V_{r_{n}}e^{isH_{0}}\psi^{j})ds\Big\|_{S(\dot{H}^{s_{c}})}\nonumber\\
&\lesssim \|V_{r_{n}}e^{itH_{0}}\psi^{j}\|_{L_{t}^{\frac{4}{2-s_{c}}}L_{x}^{\frac{2N}{N+4}}}\rightarrow 0,
\end{align}
as $n\rightarrow\infty$, where the last step follows from
\begin{align}\label{5.21}
\|V_{r_{n}}e^{itH_{0}}\psi^{j}\|_{L_{t}^{\frac{4}{2-s_{c}}}L_{x}^{\frac{2N}{N+4}}}
\lesssim \|V_{r_{n}}\|_{L^{\frac{N}{4}}}\|e^{itH_{0}}\psi^{j}\|_{L_{t}^{\frac{4}{2-s_{c}}}L_{x}^{\frac{2N}{N-4}}}
\lesssim \|V\|_{L^{\frac{N}{4}}}\|\psi^{j}\|_{\dot{H}^{s_{c}}},
\end{align}
and the condition $r_{n}\rightarrow 0$ or $\infty$. Thus ${\widetilde W}_{n}^{M}(x)$ in \eqref{5.17}
satisfies the property \eqref{5.13}.

To get \eqref{5.14}, it suffices to prove
\begin{align}\label{5.22}
\|{\widetilde W}_{n}^{M}\|_{L^{2}}^{2}=\|W_{n}^{M}\|_{L^{2}}^{2}+o_{n}(1).
\end{align}
It follows from the expression of ${\widetilde W}_{n}^{M}(x)$ \eqref{5.18} that
\begin{align}\label{5.23}
&\|{\widetilde W}_{n}^{M}\|_{L^{2}}^{2}=\|W_{n}^{M}\|_{L^{2}}^{2}
+2\sum_{j=1}^{M}\langle W_{n}^{M}, e^{-it_{n}^{j}H_{0}}\psi^{j}-e^{-it_{n}^{j}H_{r_{n}}}\psi^{j}\rangle\nonumber\\
&+2\sum_{k\neq j}\langle e^{-it_{n}^{k}H_{0}}\psi^{j}-e^{-it_{n}^{k}H_{r_{n}}}\psi^{j},
e^{-it_{n}^{j}H_{0}}\psi^{j}-e^{-it_{n}^{j}H_{r_{n}}}\psi^{j}\rangle\nonumber\\
&+\sum_{j=1}^{M}\|e^{-it_{n}^{j}H_{0}}\psi^{j}-e^{-it_{n}^{j}H_{r_{n}}}\psi^{j}\|_{L^{2}}^{2},
\end{align}
from which, we only need to show that
\begin{align}\label{5.24}
\|e^{-it_{n}^{j}H_{0}}\psi^{j}-e^{-it_{n}^{j}H_{r_{n}}}\psi^{j}\|_{L^{2}}\rightarrow 0,
\end{align}
as $n\rightarrow \infty$.

In fact,
\begin{align}\label{5.25}
\|e^{-it_{n}^{j}H_{0}}\psi^{j}-e^{-it_{n}^{j}H_{r_{n}}}\psi^{j}\|_{L^{2}}
&=\Big\|\displaystyle\int_{-t_{n}^{j}}^{0}e^{-i(t_{n}^{j}+s)H_{r_{n}}}(V_{r_{n}}e^{isH_{0}}\psi^{j})ds\Big\|_{L^{2}}\nonumber\\
&\lesssim \|V_{r_{n}}e^{itH_{0}}\psi^{j}\|_{L_{t}^{2}L_{x}^{\frac{2N}{N+4}}}\rightarrow 0,
\end{align}
as $n\rightarrow \infty$. where the last step follows from
\begin{align}\label{5.26}
\|V_{r_{n}}e^{itH_{0}}\psi^{j}\|_{L_{t}^{2}L_{x}^{\frac{2N}{N+4}}}
\lesssim \|V_{r_{n}}\|_{L^{\frac{N}{4}}}\ \|e^{itH_{0}}\psi^{j}\|_{L_{t}^{2}L_{x}^{\frac{2N}{N-4}}}
\lesssim \|V\|_{L^{\frac{N}{4}}}\|\psi^{j}\|_{L^{2}},
\end{align}
and the condition $r_{n}\rightarrow 0$ or $\infty$. Thus, we complete the proof of \eqref{5.14}.

Now we turn to \eqref{5.15}. Since
\begin{align*}
\|H_{r_{n}}^{\frac{1}{2}}f_{n}\|_{L^{2}}^{2}=\|\Delta f_{n}\|_{L^{2}}^{2}+\langle V_{r_{n}}f_{n}, f_{n}\rangle
\end{align*}
and
$$
|\langle V_{r_{n}}f_{n}, f_{n}\rangle|\lesssim \|V_{r_{n}}\|_{L^{\frac{N}{4}}}\ \|f_{n}\|_{L^{\frac{2N}{N-4}}}^{2}
\lesssim \|V\|_{L^{\frac{N}{4}}}\ \|\Delta f_{n}\|_{L^{2}}^{2},
$$
we have
\begin{align}\label{5.27}
\|H_{r_{n}}^{\frac{1}{2}}f_{n}\|_{L^{2}}^{2}=\|\Delta f_{n}\|_{L^{2}}^{2}+o_{n}(1),
\end{align}
provided that $\|\Delta f_{n}\|_{L^{2}}$ is uniformly bounded. Hence, applying \eqref{5.27} with $\phi_{n}$, $\phi^{j}$
and ${\widetilde W}_{n}^{M}$ and using the asymptotic Pythagorean expansion associated with the free linear propagator Lemma 5.3
in \cite{Guo}, we find that \eqref{5.15} can be deduced from the following expression
\begin{align}\label{5.28}
\|\Delta {\widetilde W}_{n}^{M}\|_{L^{2}}^{2}=\|\Delta W_{n}^{M}\|_{L^{2}}^{2}+o_{n}(1).
\end{align}

As in the proof of \eqref{5.22}, it suffices to prove
\begin{align}\label{5.29}
\|\Delta (e^{-it_{n}^{j}H_{0}}\psi^{j}-e^{-it_{n}^{j}H_{r_{n}}}\psi^{j})\|_{L^{2}}\rightarrow 0,
\end{align}
as $n\rightarrow \infty$. Indeed, using Proposition \ref{lem2.4}, we have
\begin{align}\label{5.30}
\|\Delta (e^{-it_{n}^{j}H_{0}}\psi^{j}-e^{-it_{n}^{j}H_{r_{n}}}\psi^{j})\|_{L^{2}}
&= \Big\|H_{0}^{\frac{1}{2}}\displaystyle\int_{-t_{n}^{j}}^{0}e^{-i(t_{n}^{j}+s)H_{r_{n}}}(V_{r_{n}}e^{isH_{0}}\psi^{j})ds\Big\|_{L^{2}}\nonumber\\
&\lesssim \big\|\nabla (V_{r_{n}}e^{isH_{0}}\psi^{j})\big\|_{L_{t}^{2}L_{x}^{\frac{2N}{N+2}}}\rightarrow 0,
\end{align}
as $n\rightarrow \infty$, where the last step follows from
\begin{align}\label{5.31}
\big\|\nabla (V_{r_{n}}e^{isH_{0}}\psi^{j})\big\|_{L_{t}^{2}L_{x}^{\frac{2N}{N+2}}}
&\lesssim \big\||x||\nabla V_{r_{n}}|\big\|_{L^{\frac{N}{4}}}\ \big\||x|^{-1}e^{isH_{0}}\psi^{j}\big\|_{L_{t}^{2}L_{x}^{\frac{2N}{N-6}}}
+\big\| V_{r_{n}}|\big\|_{L^{\frac{N}{4}}}\ \big\|\nabla e^{isH_{0}}\psi^{j}\big\|_{L_{t}^{2}L_{x}^{\frac{2N}{N-6}}}\nonumber\\
&\lesssim  \Big(\big\||x||\nabla V|\big\|_{L^{\frac{N}{4}}}+\big\| V|\big\|_{L^{\frac{N}{4}}}\Big)
\ \big\|\Delta e^{isH_{0}}\psi^{j}\big\|_{L_{t}^{2}L_{x}^{\frac{2N}{N-4}}}\nonumber\\
&\lesssim  \Big(\big\||x||\nabla V|\big\|_{L^{\frac{N}{4}}}+\big\| V|\big\|_{L^{\frac{N}{4}}}\Big)\ \|\psi^j\|_{H^{2}}.
\end{align}

Now Let's consider the other case $r_{n}=1$. Using \eqref{5.16} again gives
\begin{align}\label{5.32}
\phi_{n}(x)=\sum_{j=1}^{M}e^{-it_{n}^{j}H_{0}}\psi^{j}(x)+W_{n}^{M}(x).
\end{align}
If $t_{n}^{j}\rightarrow\infty$, by Proposition \ref{lem2.7}, there exists $\tilde{\psi}^{j}\in H^{2}({\bf R}^{n})$ such that
$\|e^{-it_{n}^{j}H_{0}}\psi^{j}-e^{it_{n}^{j}H}\tilde{\psi}^{j}\|_{H^{2}}\rightarrow 0$. If, on the other hand,
$t_{n}^{j}=0$, we set $\tilde{\psi}^{j}=\psi^{j}$. To sum up, in either case, we obtain a new profile $\tilde{\psi}^{j}$
for the given $\psi^{j}$ such that
\begin{align}\label{5.33}
\|e^{-it_{n}^{j}H_{0}}\psi^{j}-e^{-it_{n}^{j}H}\tilde{\psi}^{j}\|_{H^{2}}\rightarrow 0,
\ {\rm as}\ \ n\rightarrow+\infty.
\end{align}
In order to get the form of \eqref{5.11}, we can rewrite \eqref{5.32} as
\begin{align}\label{5.34}
\phi_{n}(x)=\sum_{j=1}^{M}e^{-it_{n}^{j}H}\tilde{\psi}^{j}(x)+{\widetilde W}_{n}^{M}(x),
\end{align}
where
\begin{align}\label{5.35}
{\widetilde W}_{n}^{M}(x)=W_{n}^{M}(x)+\sum_{j=1}^{M}\Big(e^{-it_{n}^{j}H_{0}}\psi^{j}(x)-e^{-it_{n}^{j}H}\tilde{\psi}^{j}(x)\Big).
\end{align}
Here we only give the proof of \eqref{5.13}, since all the proofs of \eqref{5.13}-\eqref{5.15} can be obtained by following the same argument
in the case $r_{n}\rightarrow 0$ or $\infty$ and using \eqref{5.33}.  Indeed, \eqref{5.19} with $r_{n}=1$ is still valid, which yields
\begin{align}\label{5.36}
\lim_{M\rightarrow+\infty}\Big(\lim_{n\rightarrow+\infty}\|e^{itH}W_{n}^{M}\|_{S(\dot{H}^{s_{c}})}\Big)=0.
\end{align}
And using the Strichartz estimate \eqref{2.7}  and \eqref{5.33}, we have
\begin{align}\label{5.37}
\|e^{itH}(e^{-it_{n}^{j}H_{0}}\psi^{j}(x)-e^{-it_{n}^{j}H}\tilde{\psi}^{j}(x))\|_{S(\dot{H}^{s_{c}})}
&\lesssim \|e^{-it_{n}^{j}H_{0}}\psi^{j}(x)-e^{-it_{n}^{j}H}\tilde{\psi}^{j}(x)\|_{\dot{H}^{s}}\nonumber\\
&\lesssim \|e^{-it_{n}^{j}H_{0}}\psi^{j}(x)-e^{-it_{n}^{j}H}\tilde{\psi}^{j}(x)\|_{H^{2}}\rightarrow 0,
\end{align}
as $n\rightarrow\infty$. putting \eqref{5.36} and \eqref{5.37} together gives \eqref{5.13}, that is,
\begin{align}\label{5.38}
\lim_{M\rightarrow+\infty}\Big(\lim_{n\rightarrow+\infty}\|e^{itH}{\widetilde W}_{n}^{M}\|_{S(\dot{H}^{s_{c}})}\Big)=0.
\end{align}
\end{proof}

\begin{remark}\label{rem5.4}
 In the linear profile decomposition \eqref{5.11}, we still have the property, for any $j\geq 1$,
\begin{align}\label{5.39}
W_{n}^{j}-e^{-t_{n}^{j}H_{r_{n}}}\psi^{j+1}\rightharpoonup0 \ \ {\rm in} \ \ H^{2}({\bf R}^{N})
\end{align}
In fact, when $r_{n}\rightarrow 0$ or $r_{n}\rightarrow\infty$, by \eqref{5.18}, we have, for any $M\geq 1$,
\begin{align}\label{5.40}
{\widetilde W}_{n}^{M}(x)=W_{n}^{M}(x)+\sum_{j=1}^{M}\Big(e^{-it_{n}^{j}H_{0}}\psi^{j}(x)-e^{-it_{n}^{j}H_{r_{n}}}\psi^{j}(x)\Big).
\end{align}
It follows from \eqref{5.24} and \eqref{5.29} that, for any $j\geq 1$,
\begin{align}\label{5.41}
e^{-it_{n}^{j}H_{0}}\psi^{j}-e^{-it_{n}^{j}H_{r_{n}}}\psi^{j}\rightarrow 0 \ \ {\rm in} \ \ H^{2}({\bf R}^{N}),
\end{align}
which together with the known result $W_{n}^{M} -e^{-t_{n}^{j}H_{0}}\psi^{M+1}\rightharpoonup 0$ in $H^{2}({\bf R}^{N})$ implies that
\begin{align*}
{\widetilde W}_{n}^{M} -e^{-t_{n}^{j}H_{0}}\psi^{M+1}\rightharpoonup 0 \ \ {\rm in} \ \ H^{2}({\bf R}^{N}).
\end{align*}
Using \eqref{5.41} with $j=M+1$ again gives
\begin{align}\label{5.42}
{\widetilde W}_{n}^{M} -e^{-t_{n}^{j}H_{r_{n}}}\psi^{M+1}\rightharpoonup 0 \ \ {\rm in} \ \ H^{2}({\bf R}^{N}),
\end{align}
which is namely our desired result \eqref{5.39}.
On the other hand, when $r_{n}=1$, by \eqref{5.35}
\begin{align}\label{5.43}
{\widetilde W}_{n}^{M}(x)=W_{n}^{M}(x)+\sum_{j=1}^{M}\Big(e^{-it_{n}^{j}H_{0}}\psi^{j}(x)-e^{-it_{n}^{j}H}\tilde{\psi}^{j}(x)\Big).
\end{align}
By \eqref{5.33}, whenever $t_{n}^{j}=0$ or $t_{n}^{j}\rightarrow\infty$,
\begin{align}\label{5.44}
e^{-it_{n}^{j}H_{0}}\psi^{j}(x)-e^{-it_{n}^{j}H}\tilde{\psi}^{j}(x)\rightarrow 0 \ \ {\rm in} \ \ H^{2}({\bf R}^{N}).
\end{align}
Similarly, we have
\begin{align}\label{5.45}
{\widetilde W}_{n}^{j} -e^{-it_{n}^{j}H}\tilde{\psi}^{j+1}\rightharpoonup 0 \ \ {\rm in} \ \ H^{2}({\bf R}^{N}).
\end{align}
\end{remark}

Next, we shall use Lemma \ref{lem5.2}, Proposition \ref{pro5.3} and Remark \ref{rem5.4} to establish the energy pythagorean expansion.

\begin{lemma}\label{lem5.5} In the situation of Proposition \ref{pro5.3}, we have
\begin{align}\label{5.46}
E_{V_{r_{n}}}(\phi_{n})=\sum_{j=1}^{M}E_{V_{r_{n}}}(e^{-it_{n}^{j}H_{r_{n}}}\psi^{j})+E_{V_{r_{n}}}(W_{n}^{M})+o_n(1).
\end{align}
\end{lemma}

\begin{proof}
 According to \eqref{5.14} and \eqref{5.15}, it suffices to establish for all $M\geq1$,
\begin{align}\label{5.47}
\big\|\phi_{n}\big\|_{p+1}^{p+1}=\sum_{j=1}^{M}\big\|e^{-it_{n}^{j}H_{r_{n}}}\psi^{j}\big\|_{p+1}^{p+1}+\big\|W_{n}^{M}\big\|_{p+1}^{p+1}+o_n(1).
\end{align}
In fact, there are only two cases to consider.

Case 1. There exists some $j$ for which $t_n^j$ converges to a finite number, which, without loss of generality,
we assume is 0. In this case we will show that $\lim_{n\rightarrow\infty}\|W_n^M\|_{p+1}=0$  for $M>j$,
 $\lim_{n\rightarrow\infty}\|e^{-it_{n}^{k}H_{r_{n}}}\psi^{k}\|_{p+1}=0$ for all $k\neq j$, and
 $\lim_{n\rightarrow\infty}\|\phi_n\|_{p+1}=\|\psi^{j}\|_{p+1}$, which gives \eqref{5.47}.

 Case 2. For all $j$, $|t_n^j|\rightarrow\infty$. In this case we will show that
 $\lim_{n\rightarrow\infty}\|e^{-it_{n}^{k}H_{r_{n}}}\psi^{k}\|_{p+1}=0$ for all $k$ and
 $\lim_{n\rightarrow\infty}\|\phi_n\|_{p+1}=\lim_{n\rightarrow\infty}\|W_n^M\|_{p+1}$, which  gives
\eqref{5.47} again.

 For Case 1: We infer from Remark \ref{rem5.4} that $W_n^{j-1}\rightharpoonup\psi^j$.
 By the compactness of the embedding $H^2_{rad}\hookrightarrow L^{p+1}$, it follows that
$W_n^{j-1}\rightarrow\psi^j$ strongly in $L^{p+1}$. Let $k\neq j$. Then we get from \eqref{5.12}
that $|t_n^k|\rightarrow\infty$. By Lemma \ref{lem5.2}, we obtain that
$\|e^{-it_{n}^{k}H_{r_{n}}}\psi^{k}\|_{p+1}\rightarrow0$. Recalling that
\begin{align}\label{5.48}
W_n^{j-1}=\phi_n-e^{-it_{n}^{1}H_{r_{n}}}\psi^{1}-\cdots -e^{-it_{n}^{j-1}H_{r_{n}}}\psi^{j-1},
\end{align}
we conclude that $\phi_n\rightarrow\psi^j$ strongly in $L^{p+1}$. Since
\begin{align}\label{5.49}
W_n^M=
W_n^{j-1}-\psi^j-e^{-it_{n}^{j+1}H_{r_{n}}}\psi^{j+1}-\cdots -e^{-it_{n}^{M}H_{r_{n}}}\psi^{M},
\end{align}
we also conclude that $\lim_{n\rightarrow\infty}\|W_n^M\|_{p+1}\rightarrow0$ strongly in $L^{p+1}$, for $M>j$.

For Case 2: Since
\begin{align}\label{5.50}
W_n^{M}=\phi_n-e^{-it_{n}^{1}H_{r_{n}}}\psi^{1}-\cdots -e^{-it_{n}^{M}H_{r_{n}}}\psi^{M},
\end{align}
and for all $j$, $|t_{n}^{j}|\rightarrow\infty$, which gives
$\lim_{n\rightarrow\infty}\|e^{-it_{n}^{j}H_{r_{n}}}\psi^{j}\|_{p+1}=0$,
we conclude that $\phi_{n}-W_{n}^{M}\rightarrow 0$ in $L^{p+1}$.
Hence, we have $\lim_{n\rightarrow\infty}\|\phi_n\|_{p+1}=\lim_{n\rightarrow\infty}\|W_n^M\|_{p+1}$.
\end{proof}

\begin{proposition}\label{pro5.6} Under the assumptions of Theorem \ref{th1.2}, then
 there exists a radial $u_{c,0}$ in $H^2({\bf R}^{N})$ with
\begin{align}\label{5.51}
M(u_{c,0})^{\frac{2-s_c}{s_c}}E(u_{c,0})= (M^{\frac{2-s_c}{s_c}}E)_c<M(Q)^{\frac{2-s_c}{s_c}}E_{0}(Q),
\end{align}
\begin{align}\label{5.52}
\| u_{c,0}\|^{\frac{2-s_c}{s_c}}_{2}\|H^{\frac{1}{2}}u_{c,0}\|_{2}
<\|Q\|^{\frac{2-s_c}{s_c}}_{2}\|\Delta Q\|_{2}
\end{align}
such that  the corresponding solution $u_c$  of \eqref{1.1} to
the initial data $u_{c,0}$  is global and $$\|u_c\|_{S(\dot{H}^{s_{c}})}=\infty.$$
\end{proposition}

\begin{proof}
 By the assumption \eqref{5.5} and the definition of $(M^{\frac{2-s_c}{s_c}}E)_c$, we can
find a sequence of solutions $u_{n}(t)={\rm BNLS_{V}}u_{n, 0}$ of \eqref{1.1} with initial data $u_{n,0}$ such that
\begin{align}\label{5.53}
M(u_{n,0})^{\frac{2-s_c}{s_c}}E(u_{n,0})\downarrow (M^{\frac{2-s_c}{s_c}}E)_c,
\end{align}
\begin{align}\label{5.54}
\| u_{n,0}\|^{\frac{2-s_c}{s_c}}_{L^{2}}\|H^{\frac{1}{2}}u_{n,0}\|_{L^{2}}
<\|Q\|^{\frac{2-s_c}{s_c}}_{L^{2}}\|\Delta Q\|_{L^{2}}
\end{align}
and
\begin{align}\label{5.55}
\|u_n\|_{S(\dot{H}^{s_{c}})}=\infty.
\end{align}
Note that it's not obvious for uniform boundedness of $\|u_{n, 0}\|_{H^{2}}$ because of
shortness of scaling invariance for the equation \eqref{1.1}. Hence, the first step is to show
that $\|u_{n, 0}\|_{H^{2}}$ is uniformly bounded, which can be obtained from the fact that
passing to a subsequence,
\begin{align}\label{5.56}
r_{n}=\|u_{n,0}\|_{L^{2}}^{-\frac{1}{s_{c}}}\sim 1.
\end{align}
Indeed, by $V\geq 0$, we have
\begin{align}\label{5.57}
\|u_{n, 0}\|_{H^{2}}^{2}&=\|u_{n, 0}\|_{L^{2}}^{2}+\|\Delta u_{n, 0}\|_{L^{2}}^{2}\nonumber\\
&\leq  \|u_{n, 0}\|_{L^{2}}^{2}+\|H^{\frac{1}{2}} u_{n, 0}\|_{L^{2}}^{2}\nonumber\\
&<r_{n}^{-2s_{c}}+\|Q\|_{L^{2}}^{\frac{4-2s_{c}}{s_{c}}}\|\Delta Q\|_{L^{2}}^{2}r_{n}^{4-2s_{c}}.
\end{align}
Let \eqref{5.56} be false, then we may assume that $r_{n}\rightarrow 0$ or $\infty$. Next, we shall apply
the linear profile decomposition and the perturbation lemma to get a contradiction. To this end,
we define
$$
\tilde{u}_{n}(x, t)=\frac{1}{r_{n}^{\frac{4}{p-1}}}u_{n}\Big(\frac{x}{r_{n}}, \frac{t}{r_{n}^{2}}\Big),
$$
and
$$
\tilde{u}_{n, 0}(x)=\frac{1}{r_{n}^{\frac{4}{p-1}}}u_{n, 0}\Big(\frac{x}{r_{n}}\Big).
$$
Hence, $\tilde{u}_{n}={\rm BNLS}_{V_{r_{n}}}\tilde{u}_{n, 0}$, that is, $\tilde{u}_{n}$ is the
solution to the initial value problem
\begin{equation}\label{5.58}
\left\{ \begin{aligned}
  i\partial_{t}\tilde{u}_{n}+H_{r_{n}}\tilde{u}_{n}-|\tilde{u}_{n}|^{p-1}\tilde{u}_{n}=0,\\
  \tilde{u}_{n}(0)=\tilde{u}_{n, 0},
                          \end{aligned}\right.
                          \end{equation}
and $\|\tilde{u}_{n, 0}\|_{H^{2}}$ is uniformly bounded, which follows from
$$
\|\tilde{u}_{n, 0}\|_{L^{2}}^{2}=r_{n}^{2s_{c}}\|u_{n, 0}\|_{L^{2}}^{2}=1
$$
and
\begin{align*}
\|\Delta\tilde{u}_{n, 0}\|_{L^{2}}^{2}
&\leq\|H_{r_{n}}^{\frac{1}{2}}\tilde{u}_{n, 0}\|_{L^{2}}^{2}
=r_{n}^{2s_{c}-4}\|H^{\frac{1}{2}}u_{n, 0}\|_{L^{2}}^{2}\nonumber\\
&= \|u_{n,0}\|_{L^{2}}^{\frac{4-2s_{c}}{s_{c}}}\|H^{\frac{1}{2}}u_{n, 0}\|_{L^{2}}^{2}
<\|Q\|_{L^{2}}^{\frac{4-2s_{c}}{s_{c}}}\|\Delta Q\|_{L^{2}}^{2}.\nonumber
\end{align*}
Therefore, we apply Proposition \ref{pro5.3} to $\tilde{u}_{n, 0}$ to get
\begin{align}\label{5.59}
\tilde{u}_{n, 0}(x)=\sum_{j=1}^{M}e^{-it_{n}^{j}H_{r_{n}}}\psi^{j}(x)+W_{n}^{M}(x).
\end{align}
Then by \eqref{5.46}, we have further
\begin{align}\label{5.60}
\sum_{j=1}^{M}\lim_{n\rightarrow\infty}E_{V_{r_{n}}}(e^{-it_{n}^{j}H_{r_{n}}}\psi^{j})
+\lim_{n\rightarrow\infty}E_{V_{r_{n}}}(W_{n}^{M})
=\lim_{n\rightarrow\infty}E_{V_{r_{n}}}(\tilde{u}_{n, 0}).
\end{align}
Since also by the profile expansion, we have
\begin{align}\label{5.61}
1=\|\tilde{u}_{n, 0}\|_{L^{2}}^{2}
=\sum_{j=1}^{M}\|\psi^{j}\|_{L^{2}}^{2}+\|W_{n}^{M}\|_{L^{2}}^{2}+o_{n}(1),
\end{align}
\begin{align}\label{5.62}
\|H_{r_{n}}^{\frac{1}{2}}\tilde{u}_{n, 0}\|_{L^{2}}^{2}
=\sum_{j=1}^{M}\|H_{r_{n}}^{\frac{1}{2}}e^{-it_{n}^{j}H_{r_{n}}}\psi^{j}\|_{L^{2}}^{2}
+\|H_{r_{n}}^{\frac{1}{2}}e^{-it_{n}^{j}H_{r_{n}}}W_{n}^{M}\|_{L^{2}}^{2}+o_{n}(1),
\end{align}
Since from the proof of Lemma \ref{lem4.4}, each energy in nonnegative and then
\begin{align}\label{5.63}
\lim_{n\rightarrow\infty}E_{V_{r_{n}}}(e^{-it_{n}^{j}H_{r_{n}}}\psi^{j})
&\leq
\lim_{n\rightarrow\infty}E_{V_{r_{n}}}(\tilde{u}_{n, 0})
=\lim_{n\rightarrow\infty}M(u_{n,0})^{\frac{2-s_c}{s_c}}E(u_{n,0})\nonumber\\
&=(M^{\frac{2-s_c}{s_c}}E)_c<M(Q)^{\frac{2-s_c}{s_c}}E_{0}(Q).
\end{align}
For a given $j$, if $|t_n^j|\rightarrow+\infty$, we may assume $t_n^j\rightarrow+\infty$
or $t_n^j\rightarrow-\infty$  up to a subsequence. In this case, by \eqref{5.61} and \eqref{5.63}
with $V=0$, we have
\begin{align}\label{5.64}
\frac{1}{2}\| \psi^{j}\|_{L^{2}}^{\frac{2-s_c}{s_c}}\|\Delta \psi^{j}\|_{L^{2}}
<M(Q)^{\frac{2-s_c}{s_c}}E_{0}(Q).
\end{align}
If we denote by ${\rm BNLS}_{0}(t)\phi$   a solution of \eqref{1.1}  with $V=0$ and initial data $\phi$,
then we get  from the existence of wave operators ( Proposition \ref{pro4.5} with $V=0$ or Proposition 4.4 in \cite{Guo} )that
there exists $\tilde{\psi}^{j}$ such that
\begin{align}\label{5.65}
\big\|{\rm BNLS}_{0}(-t_{n}^{j})\tilde{\psi}^{j}-e^{-it_{n}^{j}H_{0}}\psi^{j}\big\|_{H^2}
\rightarrow0,\ \ as\ \ n\rightarrow+\infty.
\end{align}
If, on the other hand, $t_{n}^{j}=0$, we set $\tilde{\psi}^{j}=\psi^{j}$.
To sum up, in either case, we obtain a $\tilde{\psi}^{j}$ for the given $\psi^{j}$ such that \eqref{5.65}.

In order to use the perturbation theory to get a contradiction, we set
$v^{j}(t)={\rm BNLS}_{0}(t)\tilde{\psi}^{j}$,
$v_{n}(t)=\sum_{j=1}^{M}v^j(t-t_{n}^{j})$,
and $\tilde{v}_{n}(t)={\rm BNLS}_{0}v_{n}(0)$. We will prove successively the following three claims
to get a contradiction.

~{\it Claim 1.} There exists a large constant $A_{0}$ independent of $M$ such
that there exists $n_0=n_0(M)$ such that for $n\geq n_0$,
\begin{align}\label{5.66}
\|\tilde{v}_n\|_{S(\dot{H}^{s_{c}})}\leq A_{0}.
\end{align}
Indeed, using \eqref{5.12} and \eqref{5.65}, we have that
\begin{align}\label{5.67}
E_{0}(v_{n}(0))=\sum_{j=1}^{M}E_{0}(v^{j}(-t^{j}))+o_{n}(1)
=\sum_{j=1}^{M}E_{0}(e^{-it_{n}^{j}H_{0}}\psi^{j})+o_{n}(1)
\end{align}
By \eqref{5.24}, \eqref{5.29},
the assumption $r_{n}\rightarrow 0$ or $\infty$ and Lemma \ref{lem5.5} , we have
\begin{align}\label{5.68}
&\quad \sum_{j=1}^{M}E_{0}(e^{-it_{n}^{j}H_{0}}\psi^{j})
=\sum_{j=1}^{M}E_{V_{r_{n}}}(e^{-it_{n}^{j}H_{r_{n}}}\psi^{j})+o_{n}(1)\nonumber\\
&\leq E_{V_{r_{n}}}(\tilde{u}_{n, 0})+o_{n}(1)=r_{n}^{2s_{c}-4}E(u_{n, 0})+o_{n}(1)
\end{align}
Collecting \eqref{5.67} and \eqref{5.68} gives
\begin{align}\label{5.69}
E_{0}(v_{n}(0))\leq r_{n}^{2s_{c}-4}E(u_{n, 0})+o_{n}(1)
\end{align}
Similarly, we have
\begin{align}\label{5.70}
M(v_{n}(0))\leq M(\tilde{u}_{n, 0})+o_{n}(1)=r_{n}^{2s_{c}}M(u_{n, 0})+o_{n}(1)
\end{align}
and
\begin{align}\label{5.71}
\|\Delta v_{n}(0)\|_{L^{2}}\leq \|H_{r_{n}}^{\frac{1}{2}}\tilde{u}_{n, 0}\|_{L^{2}}
=r_{n}^{s_{c}-2}\|H^{\frac{1}{2}}u_{n,0}\|_{L^{2}}
\end{align}
Hence, \eqref{5.69}-\eqref{5.71} imply for large $n$,
\begin{align*}
M(v_{n}(0))^{\frac{2-s_{c}}{s_{c}}}E_{0}(v_{n}(0))
&\leq
M(u_{n, 0})^{\frac{2-s_{c}}{s_{c}}}E(u_{n, 0})+o_{n}(1)\nonumber\\
&=(M^{\frac{2-s_c}{s_c}}E)_c+o_{n}(1)
<M(Q)^{\frac{2-s_c}{s_c}}E_{0}(Q)
\end{align*}
and
\begin{align}\label{5.72}
\|v_{n}(0)\|_{L^{2}}^{\frac{2-s_{c}}{s_{c}}}\|\Delta v_{n}(0)\|_{L^{2}}
\leq \|u_{n, 0}\|_{L^{2}}^{\frac{2-s_{c}}{s_{c}}}\|H^{\frac{1}{2}}u_{n,0}\|_{L^{2}}
+o_{n}(1)<\|Q\|^{\frac{2-s_c}{s_c}}_{2}\|\Delta Q\|_{2}
\end{align}
Hence, it follows from Theorem \ref{th1.1} that \eqref{5.66} is true.

~{\it Claim 2.}  There exists a large constant $A_{1}$ independent of $M$ such
that there exists $n_1=n_1(M)$ such that for $n\geq n_1$,
\begin{align}\label{5.73}
\|v_n\|_{S(\dot{H}^{s_{c}})}\leq A_{1}.
\end{align}
In fact, we note that
\begin{align}\label{5.74}
i\partial_t v_n+\Delta^2 v_n-|v_n|^{p-1}v_n=e_n,
\end{align}
where
\begin{align}\label{5.75}
e_n=\sum_{j=1}^{M}|v^j(t-t_{n}^{j})|^{p-1}v^j(t-t_{n}^{j})-|\sum_{j=1}^{M}v^j(t-t_{n}^{j})|^{p-1}\sum_{j=1}^{M}v^j(t-t_{n}^{j}).
\end{align}
If $p-1>1$, we estimate
\begin{align}\label{5.76}
|e_n|\leq c\sum\sum_{k\neq j}^{M}|v^j(t-t_{n}^{j})|
|v^k(t-t_{n}^{k})|\ \Big(|v^j(t-t_{n}^{j})|^{p-2}+|v^k(t-t_{n}^{k})|^{p-2}\Big);
\end{align}
while if $p-1<1$,
\begin{align}\label{5.77}
|e_n|\leq c\sum\sum_{k\neq j}^{M}|v^j(t-t_{n}^{j})|
|v^k(t-t_{n}^{k})|^{p-1}.
\end{align}
Since, for $j\neq k$, $|t_{n}^{j}-
t_{n}^{k}|\rightarrow+\infty$, then
we obtain that $\|e_n\|_{S'(\dot{H}^{-s_{c}})}$ goes to zero as $n\rightarrow\infty$, which,
combined  with \eqref{5.66} and Lemma \ref{lem2.11} with $V=0$, gives \eqref{5.73}.

~{\it Claim 3.}  There exists a large constant $A_{2}$ independent of $M$ such
that there exists $n_2=n_2(M)$ such that for $n\geq n_2$,
\begin{align}\label{5.78}
\|\tilde{u}_n\|_{S(\dot{H}^{s_{c}})}\leq A_{2}.
\end{align}
To see this, we note that
\begin{align}\label{5.79}
i\partial_t v_n+H_{r_{n}} v_n-|v_n|^{p-1}v_n={\tilde e}_n,
\end{align}
where
\begin{align}\label{5.80}
{\tilde e}_n=V_{r_{n}}v_{n}+e_{n}.
\end{align}
We will use the perturbation theory to get \eqref{5.78}. To this end, we will control two norms, that is,
\begin{align}\label{5.81}
\|e^{itH_{r_{n}}}(\tilde{u}_{n, 0}-v_{n}(0))\|_{S(\dot{H}^{s_{c}})}\ \ {\rm and  }\ \
\|{\tilde e}_n\|_{S'(\dot{H}^{-s_{c}})}.
\end{align}
From \eqref{5.59} and the definition of $v_{n}(t)$, we have
\begin{align}\label{5.82}
\tilde{u}_{n, 0}-v_{n}(0)=W_{n}^{M}+\sum_{j=1}^{M}\Big(e^{-it_{n}^{j}H_{r_{n}}}\psi^{j}-v^{j}(-t_{n}^{j})\Big).
\end{align}
 Let $\epsilon_0=\epsilon_0(A_{2},n,p)$ be
a small number given in Lemma \ref{lem2.11}. By \eqref{5.13}, takeing $M$ large enough such that
there exists $n_{3}=n_{3}(M)$ satisfying
\begin{align}\label{5.83}
\|e^{itH_{r_{n}}}W_{n}^{M}\|_{S(\dot{H}^{s_{c}})}< \frac{\epsilon_{0}}{2}
\end{align}
for all $n\geq n_{3}$. Next we turn to the estimate of
\begin{align}\label{5.84}
\big\|e^{itH_{r_{n}}}(e^{-it_{n}^{j}H_{r_{n}}}\psi^{j}-v^{j}(-t_{n}^{j}))\big\|_{S(\dot{H}^{s_{c}})}
\end{align}
for each $j$.
From the triangle inequality, Strichartz estimates, \eqref{5.41} and \eqref{5.65}, it follows that
there exists $n_{4}=n_{4}(M)$ such that for each $j$ and $n\geq n_{4}$
\begin{align}\label{5.85}
\|e^{itH_{r_{n}}}(e^{-it_{n}^{j}H_{r_{n}}}\psi^{j}-v^{j}(-t_{n}^{j}))\|_{S(\dot{H}^{s_{c}})}< \frac{\epsilon_{0}}{2M}.
\end{align}
~From \eqref{5.83} and \eqref{5.85}, it follows that
\begin{align}\label{5.86}
\|e^{itH_{r_{n}}}(\tilde{u}_{n, 0}-v_{n}(0))\|_{S(\dot{H}^{s_{c}})}<\epsilon_{0}
\end{align}
for all $n\geq\max\{ n_{3}, n_{4}\}$.

Similar to the proof of \eqref{5.20} and using \eqref{5.73}, we have that
$\|V_{r_{n}}v_{n}\|_{S'(\dot{H}^{-s_{c}})}$ goes to zero as $n\rightarrow\infty$,
which together with $\lim_{n\rightarrow\infty}\|e_n\|_{S'(\dot{H}^{-s_{c}})}=0$ gives
\begin{align}\label{5.87}
\lim_{n\rightarrow\infty}\|\tilde{e}_n\|_{S'(\dot{H}^{-s_{c}})}=0.
\end{align}
Applying Lemma \ref{lem2.11} with \eqref{5.86}, \eqref{5.87} and \eqref{5.73}, we get \eqref{5.78}.

By scaling, we have
\begin{align}\label{5.88}
\|u_n\|_{S(\dot{H}^{s_{c}})}=\|\tilde{u}_n\|_{S(\dot{H}^{s_{c}})}\leq A_{2},
\end{align}
 contradicting \eqref{5.55}. So $\|u_{n, 0}\|_{H^{2}}$ is uniformly bounded.

The next step is to extract $u_{c,0}$ from a bounded sequence $\{u_{n, 0}\}_{n=1}^{+\infty}$.
We omit the proof because it is similar to the proof of Proposition 5.5 in \cite{Guo}. Indeed, it
suffices to replace $e^{-itH_{0}}$ by $e^{-itH}$ in the proof.
\end{proof}

Once we established Proposition \ref{pro5.6}, we can obtain the following results of precompactness and
uniform localization of the minimal blow-up solution, the proof of which is standard and we omit
here.

\begin{proposition}\label{pro5.7}
 Let $u_c$ be
as  in Proposition \ref{pro5.6}. Then
$$
K=\Big\{u_c(t)| ~t\in{\bf R}\Big\}\subset H^2({\bf R}^{N})
$$
is precompact in $H^2({\bf R}^{N})$.
\end{proposition}

\begin{coro}\label{cor5.8}
 Let $V$, $p$ and $N$ satisfy the assumptions of Theorem \ref{th1.2}. Suppose that $u$ be a solution of \eqref{1.1} such that $K=\{u(t)|
~t\in {\bf R} \}$ is precompact in $H^2({\bf R}^{N})$. Then for each
$\epsilon>0,$ there exists $R>0$ independent of $t$ such that, for any $1\leq i, j\leq N$,
\begin{align}\label{5.89}
\int_{|x|>R}|\partial_{ij} u(x,t)|^2+|\partial_{j}u(x,t)|^2+|u(x,t)|^2+|u(x,t)|^{p+1}dx\leq\epsilon.
\end{align}
\end{coro}

\section{Proof of Theorem \ref{th1.2}}
\setcounter{equation}{0}

In this section, we prove the following rigidity statement and finish the proof of Theorem \ref{th1.2}.

\begin{theorem}\label{th6.1}
 Suppose that $u_0\in H^2({\bf R}^{N})$ is radial,
$$M(u_{0})^{\frac{2-s_c}{s_c}}E(u_0)<M(Q)^{\frac{2-s_c}{s_c}}E_{0}(Q)$$ and
$$\|u_{0}\|^{\frac{2-s_c}{s_c}}_{L^{2}}\|H^{\frac{1}{2}} u_{0}\|_{L^{2}}<\|Q\|^{\frac{2-s_c}{s_c}}_{L^{2}}\|\Delta Q\|_{L^{2}}.$$ Let $u$
be the corresponding solution of the equation \eqref{1.1} of Theorem \ref{th1.2} with initial data
$u_0$. If $K_+=\{u(t):t\in[0,\infty)\}$ is precompact in $H^2({\bf R}^{N})$, then
$u_0\equiv0$. The same conclusion holds if
$K_-=\{u(t):t\in(-\infty,0]\}$ is precompact in $H^2({\bf R}^{N})$.
\end{theorem}

\begin{proof}
 We first define
\begin{align}\label{6.1}
M_{a}(t)=2\displaystyle\int_{{\bf R}^{N}}\partial_{j}a \Im(\bar{u}\partial_{j}u)dx,
\end{align}
where $a\in C_{c}^{\infty}({\bf R}^{N})$.
The direct computation yields ( see e.g. Pausader \cite{Pau3})
\begin{align}\label{6.2}
M_{a}'(t)
&=2\displaystyle\int_{{\bf R}^{N}}\Big(2\partial_{j}u\partial_{k}\bar{u}\partial_{jk}\Delta a
-\frac{1}{2}\Delta^{3}a|u|^{2}-4\partial_{jk}a\partial_{ik}u\partial_{ij}\bar{u}
+\Delta^{2}a|\nabla u|^{2}\Big)dx\nonumber\\
&+\frac{2(p-1)}{p+1}\displaystyle\int_{{\bf R}^{N}}\Delta a|u|^{p+1}dx
+2\displaystyle\int_{{\bf R}^{N}}\nabla a\cdot\nabla V|u|^{2}dx,
\end{align}
Take a radially symmetric function $\phi\in C_{c}^{\infty}$ such that
$\phi(x)=|x|^{2}$ for $|x|\leq 1$ and $\phi(x)=0$ for $|x|\geq 2$, and define
$a(x)=R^{2}\phi(\frac{x}{R})$. By the repulsiveness assumption on the potential
$V$, direct computation gives
\begin{align}\label{6.3}
-M_{a}'(t)
&=16\displaystyle\int_{{\bf R}^{N}}|\partial_{ij}u|^{2}dx
-\frac{4n(p-1)}{p+1}\displaystyle\int_{{\bf R}^{N}}|u|^{p+1}dx
-4\displaystyle\int_{{\bf R}^{N}}x\cdot\nabla V|u|^{2}dx
+({\rm Remainder})\nonumber\\
&\geq  16\displaystyle\int_{{\bf R}^{N}}|\Delta u|^{2}dx
-\frac{4n(p-1)}{p+1}\displaystyle\int_{{\bf R}^{N}}|u|^{p+1}dx
+({\rm Remainder}),
\end{align}
where
\begin{align}\label{6.4}
{\rm (Remainder)}&=-16\displaystyle\int_{|x|\geq R}|\partial_{ij}u|^{2}dx
+8\displaystyle\int_{R\leq |x|\leq 2R}(\partial_{jk}\phi)\Big(\frac{x}{R}\Big)\partial_{ik}u\partial_{ij}\bar{u}dx\nonumber\\
&+\frac{4n(p-1)}{p+1}\displaystyle\int_{|x|\geq R}|u|^{p+1}dx
-\frac{2(p-1)}{p+1}\displaystyle\int_{R\leq |x|\leq 2R}(\Delta \phi)\Big(\frac{x}{R}\Big)|u|^{p+1}dx\nonumber\\
&+4\displaystyle\int_{|x|\geq R}x\cdot\nabla V|u|^{2}dx
-2\displaystyle\int_{R\leq |x|\leq 2R}R(\nabla\phi)\Big(\frac{x}{R}\Big)\cdot\nabla V|u|^{2}dx\nonumber\\
&-\frac{4}{R^{2}}\displaystyle\int_{R\leq |x|\leq 2R}
\partial_{j}u\partial_{k}\bar{u}(\partial_{jk}\Delta\phi)\Big(\frac{x}{R}\Big)dx
+\frac{1}{R^{4}}\displaystyle\int_{R\leq |x|\leq 2R}(\Delta^{3}\phi)\Big(\frac{x}{R}\Big)|u|^{2}dx.
\end{align}
From Corollary \ref{cor5.8}, we can infer that $({\rm Remainder})\rightarrow 0$ as $R\rightarrow\infty$ uniformly in
$t\in [0,\infty)$. In fact,
\begin{align}\label{6.5}
({\rm Remainder})
&\lesssim \displaystyle\int_{|x|\geq R}|\partial_{ij}u|^{2}dx
+\displaystyle\int_{|x|\geq R}|u|^{p+1}dx
+\frac{1}{R^{2}}\displaystyle\int_{|x|\geq R}|\partial_{i}u|^{2}dx\nonumber\\
&+\frac{1}{R^{4}}\displaystyle\int_{|x|\geq R}|u|^{2}dx
+\||x||\nabla V|\|_{L^{\frac{n}{4}}}\|u\|_{L^{\frac{2n}{n-4}}(|x|\geq R)}^{2}
\rightarrow 0.
\end{align}

Let a positive constant $\delta\in(0,1)$ be such that
$M(u_{0})^{\frac{2-s_c}{s_c}}E(u_0)<(1-\delta)M(Q)^{\frac{2-s_c}{s_c}}E(Q)$.
 By
Lemma \ref{lem4.3}, Remark \ref{rem1.3} and Lemma \ref{lem4.4}, we obtain that there
exists some constant $\delta_0>0$ such that
\begin{align}\label{6.6}
4\int|\Delta
u|^2dx-\frac{N(p-1)}{p+1}\int|u|^{p+1}dx
\geq\delta_0\displaystyle\int_{{\bf R}^{N}}|\Delta u_0|^2dx,
\end{align}
which implies by \eqref{6.3} and \eqref{6.5} that
\begin{align}\label{6.7}
-M_{a}'(t)\geq \delta_0\displaystyle\int_{{\bf R}^{N}}|\Delta u_0|^2dx.
\end{align}
Thus, we have
\begin{align}\label{6.8}
M_{a}(0)-M_{a}(t)\geq \delta_0 t\displaystyle\int_{{\bf R}^{N}}|\Delta u_0|^2dx.
\end{align}
On the other hand, by the definition of $M_{a}(t)$, we should have
\begin{align}\label{6.9}
|M_{a}(t)|
&\leq R\|u\|_{L^{2}}\|\nabla u\|_{L^{2}}
\lesssim R\|u\|_{L^{2}}^{\frac{3}{2}}\|\Delta u\|_{L^{2}}\nonumber\\
&\lesssim R\|u\|_{L^{2}}^{\frac{3}{2}}\|H^{\frac{1}{2}} u\|_{L^{2}}
\lesssim R\|Q\|_{H^{2}}^{\frac{1}{s_{c}}},
\end{align}
which is a contradiction for $t$ large unless $u_{0}=0$.
\end{proof}

Now, we can finish the proof of Theorem \ref{th1.2}.\\
{\bf The  Proof of Theorem \ref{th1.2}.}
In view of Proposition \ref{pro5.7}, Theorem \ref{th6.1}
implies that $u_c$ obtained in Proposition \ref{pro5.6} cannot
exist. Thus, there must holds that $(M^{\frac{2-s_c}{s_c}}E)_c
=E_{0}(Q)M(Q)^{\frac{2-s_c}{s_c}}$, which combined with
Proposition \ref{pro2.10} implies Theorem \ref{th1.2}.
$\hfill\Box$

\section{Finite-time blowup}

In this section, we prove the finite-time blowup for radial data in $H^2({\bf R}^N)$, that is, Theorem \ref{thblowup}.
To this end, we first obtain the localized virial identity using the commutator identities introduced by Boulenger and Lenzmann \cite{B-Lenzmann}.

Let $\phi:{\bf R}^N\rightarrow\bf R$ be a radial function with regularity property $\nabla^j\phi\in L^\infty({\bf R}^N)$ for $1\leq j\leq6$ and such that
\begin{align}
\phi(r)=\left\{ \begin{aligned}
  &\frac{r^2}2\;\;&r\leq1, \\
  &const. \;\; &r\geq10
 \end{aligned}\right.
 \ \ \ \ \ \ \
 \phi''(r)\leq1\ for\ r\geq0.
\end{align}
For $R>0$ given, we define the rescaled function $\phi_R:{\bf R}^N\rightarrow\bf R$ by setting
\begin{align}
\phi_R(r):=R^2\phi\left(\frac rR\right).
\end{align}
It can be checked that for all $r\geq0$,
\begin{align}
1-\phi''_R(r)\geq0,\ \ 1-\frac{\phi'_R(r)}r\geq0,\ \ N-\Delta\phi_R(r)\geq0.
\end{align}
Moreover, we also recall the following properties of $\phi_R$:
\begin{align*}
\left\{\begin{aligned}
 &\nabla\phi_R(r)=R\phi'\left(\frac rR\right)\frac x{|x|}=\left\{ \begin{aligned}
  &x\;\;&r\leq R, \\
  &0 \;\; &r\geq10R
 \end{aligned}\right.\\
 &\|\nabla^j\phi_R\|_{L^\infty}\lesssim R^{2-j},\ \ 0\leq j\leq6,\\
 &supp(\nabla^j\phi_R)\subset
 \left\{ \begin{aligned}
  &\{|x|\leq10R\}\;\;&j=1,2, \\
  &\{R\leq|x|\leq10R\} \;\; &3\leq j\leq6.
 \end{aligned}\right.\\
\end{aligned}\right.
\end{align*}
For $u\in H^2({\bf R}^N)$, we define {\it the localized virial} of $u$ to be the quantity
\begin{align}
\mathcal M_R(u):=\langle u,\Gamma_Ru\rangle=2Im\int_{\bf R^N}u\nabla\phi_R\cdot\nabla\bar udx,
\ \ \ \Gamma_R:=i\big(\nabla\phi_R\cdot\nabla+\nabla\cdot\nabla\phi_R\big).
\end{align}
By the Cauchy-Schwarz inequality, we have $|\mathcal M_R(u)|\lesssim R\|u\|_{L^2}\|\nabla u\|_{L^2}$.
\begin{lemma}\label{lemMR}
Let $N\geq2$ and $R>0$. Suppose that $u\in C([0,T);H^2({\bf R}^N))$ is a radial solution of \eqref{1.1}. Then for any $t\in[0,T)$, we have that differential
inequality
\begin{align*}
\frac d{dt}\mathcal M_R(u(t))&\leq
2N(p-1)E(u_0)-((p-1)N-8)\int_{{\bf R}^N}|H^{\frac{1}{2}} u|^2dx\\
&-\int_{{\bf R}^N}| u|^2(2x\cdot\nabla V(x)+8V(x))dx\\
&+\mathcal O\left(R^{-4}+R^{-2}\|\nabla u\|_{L^2}^2+R^{-\frac{(N-1)(p-1)}{2}}\|\nabla u\|_{L^2}^{\frac{p-1}2}+\|u\|_{L^2(|x|>R)}^2
\right).
\end{align*}
\end{lemma}
\begin{proof}
We follow the calculating in the proof of \cite[Lemma 3.1]{B-Lenzmann} and only sketch the steps except those involving the potential function $V$.

{\it Step 1}. By taking the time derivative and the equation \eqref{1.1},
\begin{align}
\frac d{dt}\mathcal M_R(u(t))=\mathcal A_1(u(t))+\mathcal A_2(u(t))+\mathcal B(u(t)
\end{align}
with
\begin{align*}
\mathcal A_1(u(t))=\big\langle u(t),[i\Gamma_R,\Delta^2]u(t)\big\rangle,\ \
\mathcal A_2(u(t))=\big\langle u(t),[i\Gamma_R,V(x)]u(t)\big\rangle,\ \
\mathcal B(u(t)=\big\langle u(t),[|u|^{p-1},i\Gamma_R]u(t)\big\rangle,
\end{align*}
where $[A,B]=AB-BA$.

{\it Step 2}. Following the proof of (3.13) on page 515 of \cite{B-Lenzmann}, for the dispersive part $\mathcal A_1$, we have
\begin{align*}
[i\Gamma_R,\Delta^2]=8\partial_{kl}^2(\partial_{lm}^2\phi_R)\partial_{mk}^2+4\partial_k(\partial_{kl}^2\Delta\phi_R)\partial_l
+2\partial_k(\Delta^2\phi_R)\partial_k+\Delta^3\phi_R.
\end{align*}
Since for a radial function $f$, and with $r=|x|$,
\begin{align*}
&\partial_{kl}^2f=\left(\delta_{kl}-\frac{x_kx_l}{r^2}\right)\frac{\partial_rf}r+\frac{x_kx_l}{r^2}\partial^2_rf,\\
&\int_{{\bf R}^N}|\Delta u|^2dx=\int_{{\bf R}^N} |\partial_r^2u|^2+\frac{N-1}{r^2}|\partial_ru|^2dx,
\end{align*}
then we have
\begin{align*}
\mathcal A_1(u(t))\leq8\int_{{\bf R}^N}|\Delta u|^2dx+\mathcal O\Big(R^{-4}+R^{-2}\|\nabla u\|_{L^2}^2\Big).
\end{align*}

{\it Step 3}. By straight calculation,
\begin{align*}
\mathcal A_2(u(t))=\langle u(t),[i\Gamma_R,V(x)]u(t)\rangle=-2\int_{{\bf R}^N}\nabla\phi_R\cdot\nabla V|u|^2dx.
\end{align*}
Thus, from the properties of $\phi_R$ and the decay of $V$, we get easily that
\begin{align*}
\mathcal A_2(u(t)&)=-2\int_{{\bf R}^N}x\cdot\nabla V|u|^2dx+2\int_{|x|\geq R}(x\cdot\nabla V-\nabla\phi_R\cdot\nabla V)|u|^2dx\\
&\leq -2\int_{{\bf R}^N}x\cdot\nabla V|u|^2dx+C\|u\|_{L^2(|x|\geq R)}^2.
\end{align*}

{\it Step 4}. For the nonlinear term $\mathcal B$, the same calculation as the step 3 on page 516 of \cite{B-Lenzmann} gives that
\begin{align*}
\mathcal B(u(t)=-\frac{2(p-1)N}{p+1}\int_{{\bf R}^N}|u|^{p+1}dx+\mathcal O\Big(R^{-\frac{(N-1)(p-1)}{2}}\|\nabla u\|_{L^2}^{\frac{p-1}2}\Big).
\end{align*}

Finally, we deduce that
\begin{align*}
\frac d{dt}\mathcal M_R(u(t))&\leq8\int_{{\bf R}^N}|\Delta u|^2dx
-2\int_{{\bf R}^N}x\cdot\nabla V|u|^2dx-\frac{2(p-1)N}{p+1}\int_{{\bf R}^N}|u|^{p+1}dx\\
&+\mathcal O\left(R^{-4}+R^{-2}\|\nabla u\|_{L^2}^2+R^{-\frac{(N-1)(p-1)}{2}}\|\nabla u\|_{L^2}^{\frac{p-1}2}+\|u\|_{L^2(|x|>R)}^2
\right)\\
&=
2N(p-1)E(u_0)-((p-1)N-8)\int_{{\bf R}^N}|H^{\frac{1}{2}} u|^2dx\\
&-\int_{{\bf R}^N}| u|^2(2x\cdot\nabla V(x)+8V(x))dx\\
&+\mathcal O\left(R^{-4}+R^{-2}\|\nabla u\|_{L^2}^2+R^{-\frac{(N-1)(p-1)}{2}}\|\nabla u\|_{L^2}^{\frac{p-1}2}+\|u\|_{L^2(|x|>R)}^2
\right)
\end{align*}
and this completes the proof of Lemma \ref{lemMR}.
\end{proof}

In the end, we will proof Theorem \ref{thblowup}.

{\bf Proof of Theorem \ref{thblowup}:}

{\it Case 1}: $E(u_0)<0$.

Setting $\delta=\big((p-1)N-8\big)/2$, then $\delta>0$ from $ p>1+{8\over N}$. From Lemma \ref{lemMR}, we obtain that
\begin{align*}
\frac d{dt}\mathcal M_R(u(t))&\leq
2N(p-1)E(u_0)-2\delta\int_{{\bf R}^N}|H^{\frac{1}{2}} u|^2dx\\
&-\int_{{\bf R}^N}| u|^2(2x\cdot\nabla V(x)+8V(x))dx\\
&+\mathcal O\left(R^{-4}+R^{-2}\|\nabla u\|_{L^2}^2+R^{-\frac{(N-1)(p-1)}{2}}\|\nabla u\|_{L^2}^{\frac{p-1}2}+\|u\|_{L^2(|x|>R)}^2
\right)\\
&\leq 2N(p-1)E(u_0)-2\delta\int_{{\bf R}^N}|\Delta u|^2dx
+\|W_-\|_{L^{\frac N4}}\|\Delta u\|_{L^2}^2\\
&+\mathcal O\left(R^{-4}+R^{-2}\|\nabla u\|_{L^2}^2+R^{-\frac{(N-1)(p-1)}{2}}\|\nabla u\|_{L^2}^{\frac{p-1}2}+\|u\|_{L^2(|x|>R)}^2
\right)\\
&=2N(p-1)E(u_0)-2\delta\int_{{\bf R}^N}|\Delta u|^2dx+\|W_-\|_{L^{\frac N4}}\|\Delta u\|_{L^2}^2\\
&+\mathcal O\left(R^{-4}+R^{-2}\|\Delta u\|_{L^2}+R^{-\frac{(N-1)(p-1)}{2}}\|\Delta u\|_{L^2}^{\frac{p-1}4}+\|u\|_{L^2(|x|>R)}^2
\right)
\end{align*}
where we use the assumption  $2x\cdot\nabla V+(p-1)NV=W_+-W_-$ with $W_-\in L^{\frac N4}$, the H\"older inequality, the Sobolev embedding  and $\|\nabla u\|_{L^2}\leq C(u_0)\|\Delta u\|_{L^2}^{\frac12}$.

Since $p-1\leq8$ and $E(u_0)<0$, we can choose $R$ sufficiently large such that for $t\in[0,T)$,
\begin{align*}
\frac d{dt}\mathcal M_R(u(t))\leq
2N(p-1)E(u_0)-\frac{3\delta}2\int_{{\bf R}^N}|\Delta u|^2dx+\|W_-\|_{L^{\frac N4}}\|\Delta u\|_{L^2}^2.
\end{align*}
And if we suppose $\|W_-\|_{L^{\frac N4}}$ is sufficiently small ( e.g. $\|W_-\|_{L^{\frac N4}} \le \delta/2$ ), then it follows that
\begin{align*}
\frac d{dt}\mathcal M_R(u(t))\leq
2N(p-1)E(u_0)-\delta\int_{{\bf R}^N}|\Delta u|^2dx,
\end{align*}
which, combined with the Cauchy-Schwarz inequality
$|\mathcal M_R(u(t))|\lesssim C(u_0)R\|\Delta u\|_{L^2}^{\frac12}$ and by elementary  analysis (see the case 1 on page 517 of \cite{B-Lenzmann}), gives that $\mathcal M_R(u(t))\rightarrow-\infty$ as $t\rightarrow t_*$ for some
finite time $t_*<+\infty$. Therefore, $u(t)$ cannot exist for all $t\geq0$. By blowup alternative for the Energy-subcritical case, this completes the proof of Theorem
\ref{thblowup}.

{\it Case 2}. $E(u_0)\geq0$,
\begin{align*}
M(u_{0})^{\frac{2-s_{c}}{s_{c}}}E(u_{0})<M(Q)^{\frac{2-s_{c}}{s_{c}}}E_{0}(Q),
\end{align*}
and
\begin{align*}
\|u_{0}\|_{L^{2}({\bf R}^{N})}^{\frac{2-s_{c}}{s_{c}}}\|H^{\frac{1}{2}} u_{0}\|_{L^{2}({\bf R}^{N})}
>\|Q\|_{L^{2}({\bf R}^{N})}^{\frac{2-s_{c}}{s_{c}}}\|\Delta Q\|_{L^{2}({\bf R}^{N})}.
\end{align*}
In this case, if we take some $\eta>0$ such that
\begin{align*}
M(u_{0})^{\frac{2-s_{c}}{s_{c}}}E(u_{0})<(1-\eta)M(Q)^{\frac{2-s_{c}}{s_{c}}}E_{0}(Q),
\end{align*}
then we actually could obtain (see the case 3 on page 518-519 of \cite{B-Lenzmann} or Theorem 4.1 of \cite{Guo}) that for $\delta=((p-1)N-8)/2$,
$$2\delta(1-\eta)\|H^{\frac{1}{2}} u\|_{L^2}^2\geq2(p-1)NE(u_0).$$
Therefore, from Lemma \ref{lemMR}, Remark \ref{rem1.3} and the previous discussion we deduce the upper bound
\begin{align*}
\frac d{dt}\mathcal M_R(u(t))&\leq
2N(p-1)E(u_0)-2\delta\int_{{\bf R}^N}|H^{\frac{1}{2}} u|^2dx\\
&-\int_{{\bf R}^N}| u|^2(2x\cdot\nabla V(x)+8V(x))dx\\
&+O\left(R^{-4}+R^{-2}\|\Delta u\|_{L^2}+R^{-\frac{(N-1)(p-1)}{2}}\|\Delta u\|_{L^2}^{\frac{p-1}4}+\|u\|_{L^2(|x|>R)}^2
\right)\\
&\leq-\Big(\frac{\delta\eta}{2}+o_R(1)\Big)\int_{{\bf R}^N}|\Delta u|^2dx+o_R(1).
\end{align*}
Hence, by choosing $R>0$ sufficiently large, we conclude that
\begin{align*}
\frac d{dt}\mathcal M_R(u(t))\leq
-\frac{\delta\eta}4\int_{{\bf R}^N}|\Delta u|^2dx.
\end{align*}
Following  case 1, $u(t)$ blows up in finite time, concluding  the proof of Theorem \ref{thblowup}.
 $\hfill\Box$

\end{document}